\documentclass[english]{amsart}

\usepackage{amsthm,amsmath}
\usepackage[utf8]{inputenc} 

\def\includegraphics{}

\usepackage{amssymb}
\usepackage{amsfonts}
\usepackage[english]{babel}
\usepackage{latexsym}  
\usepackage{longtable} 
\usepackage{array} 
\usepackage{setspace}
\usepackage{fancyhdr}
\usepackage{epsfig}
\usepackage{lscape}
\usepackage{pdflscape}
\usepackage{datetime}
\usepackage{ifpdf}
\usepackage{graphicx}
\usepackage{textcomp}
\usepackage{eufrak}
\usepackage{euscript}

\usepackage{pslatex}          
\usepackage[pdftex,breaklinks]{hyperref} 

\newtheorem{theorem}{Theorem}[section]
\newtheorem{lemma}[theorem]{Lemma}
\newtheorem{proposition}[theorem]{Proposition}
\newtheorem{corollary}[theorem]{Corollary}

\theoremstyle{definition}
\newtheorem{definition}[theorem]{Definition}
\newtheorem{remark}[theorem]{Remark}

\numberwithin{equation}{section}

\def \eN{\mbox{E.~Ntienjem}}
\def \waS{\mbox{W.~A.~Stein}}
\def \mN{\mbox{M.~Newman}}
\def \gL{\mbox{G.~Ligozat}}
\def \gK{\mbox{G.~K\"{o}hler}}
\def \ljpK{\mbox{L.~J.~P.~Kilford}}
\def \aA{\mbox{A.~Alaca}}
\def \sA{\mbox{\c{S}.~Alaca}}
\def \aP{\mbox{A.~Pizer}}
\def \ksW{\mbox{K.~S.~Williams}}
\def \jgH{\mbox{J.~G.~Huard}}
\def \gaL{\mbox{G.~A.~Lomadze}}
\def \mL{\mbox{M.~Lemire}}
\def \sC{\mbox{S.~Cooper}}
\def \pcT{\mbox{P.~C.~Toh}}
\def \bR{\mbox{B.~Ramakrishnan}}
\def \bS{\mbox{B.~Sahu}}
\def \mB{\mbox{M.~Besge}}
\def \dY{\mbox{D.~Ye}}
\def \jwlG{\mbox{J.~W.~L.~Glaisher}}
\def \sR{\mbox{S.~Ramanujan}}
\def \yK{\mbox{Y.~Kesicio$\check{g}$lu}}
\def \eR{\mbox{E.~Royer}}
\def \zmO{\mbox{Z.~M.~Ou}}
\def \bkS{\mbox{B.~K.~Spearman}}
\def \hhC{\mbox{H.~H.~Chan}}
\def \exwX{\mbox{E.~X.~W.~Xia}}
\def \xlT{\mbox{X.~L.~Tian}}
\def \oxmY{\mbox{O.~X.~M.~Yao}}
\def \nK{\mbox{N.~Koblitz}}
\def \tM{\mbox{T.~Miyake}}

\def \E{\mbox{$\EuFrak{E}$}}
\def \S{\mbox{$\EuFrak{S}$}}
\def \M{\mbox{$\EuFrak{M}$}}

\hypersetup{
pdfinfo={
Author={Eb\'{e}n\'{e}zer Ntienjem},
   Title={Elementary Evaluation of Convolution Sums involving the Sum of 
Divisors Function for a Class of positive Integers},
   CreationDate={D:20160407191900},
   ModDate={D:\pdfdate},
   Subject={2010 Mathematics Subject Classification. 11A25, 11E20, 11E25,
     11F11, 11F20, 11F27},
   Keywords={Sums of Divisors function; Dedekind eta function; Convolution
     Sums; Modular Forms; Eisenstein series; Eisenstein forms; Cusp Forms;
     Octonary quadratic Forms; Number of Representations}
}
}

\author{Eb{\'{e}}n{\'{e}}zer Ntienjem}

\address{
Centre for Research in Algebra and Number Theory \\
School of Mathematics and Statistics \\
Carleton University \\ 
Ottawa, Ontario, Canada, K1S 5B6
}
\email{ebenezer.ntienjem@carleton.ca;ntienjem@gmail.com} 



\begin{document}

\title[Elementary Evaluation of Convolution Sums for a Class of positive Integers]
{Elementary Evaluation of Convolution Sums involving  
the Sum of Divisors Function for a Class of positive Integers
}


\begin{abstract}
We discuss an elementary method for the evaluation of the convolution sums $\underset{\substack{
 {(l,m)\in\mathbb{N}_{0}^{2}} \\ {\alpha\,l+\beta\,m=n}
} }{\sum}\sigma(l)\sigma(m)$ for those $\alpha,\beta\in\mathbb{N}$ 
for which $\gcd{(\alpha,\beta)}=1$ and $\alpha\beta=2^{\nu}\mho$,  
where $\nu\in\{0,1,2,3\}$ and $\mho$ is 
a finite product of distinct odd primes.  Modular forms are used to achieve this 
result. We also generalize the extraction of the convolution sum to all natural 
numbers. Formulae for the number of representations of a positive integer 
$n$ by octonary quadratic forms using convolution sums belonging 
to this class are then determined when 
$\alpha\beta\equiv 0\pmod{4}$ or $\alpha\beta\equiv 0\pmod{3}$. To achieve this application, 
we first discuss a method to compute all pairs 
$(a,b),(c,d)\in\mathbb{N}^{2}$ necessary for the determination of 
such formulae for the number of representations of a positive integer 
$n$ by octonary quadratic forms when $\alpha\beta$ has the above  
form and $\alpha\beta\equiv 0\pmod{4}$ or $\alpha\beta\equiv 0\pmod{3}$. 
We illustrate our approach by explicitly evaluating the convolution sum for 
$\alpha\beta=33=3\cdot 11,\> \alpha\beta=40=2^{3}\cdot 5$ and   
$\alpha\beta=56=2^{3}\cdot 7$, and by revisiting the evaluation of the convolution 
sums for $\alpha\beta=10$, $11$, $12$, $15$, $24$.
We then apply these convolution sums to determine formulae for the number of 
representations of a positive integer $n$ by octonary quadratic forms. 
In addition, we determine formulae for the number of representations of a positive 
integer $n$ when $(a,b)=(1,1)$, $(1,3)$, $(2,3)$, $(1,9)$. 
\end{abstract}

\maketitle




\section{Introduction} \label{introduction}

We denote by $\mathbb{N}$, $\mathbb{N}_{0}$, $\mathbb{Z}$, $\mathbb{Q}$, 
$\mathbb{R}$ and $\mathbb{C}$ the sets of
positive integers, nonnegative integers, integers, rational numbers, 
real numbers and complex numbers, respectively. 

Suppose that $k\in\mathbb{N}_{0}$ and $n\in\mathbb{N}$. The sum of 
positive divisors of $n$ 
to the power of $k$, $\sigma_{k}(n)$, is defined by  
\begin{equation} \label{def-sigma_k-n}
 \sigma_{k}(n) = \sum_{0<\delta|n}\delta^{k}.
 \end{equation} 
It is obvious from the definition that $\sigma_{k}(m)=0$ for all 
$m\notin\mathbb{N}$.  
We  write $d(n)$ and $\sigma(n)$ as a 
shorthand for $\sigma_{0}(n)$ and $\sigma_{1}(n)$, respectively. 

Assume that the positive integers $\alpha\leq\beta$ are given. Then 
the convolution sum $W_{(\alpha,\beta)}(n)$ is defined by 
 \begin{equation} \label{def-convolution_sum}
 W_{(\alpha, \beta)}(n) =  \sum_{\substack{
 {(l,m)\in\mathbb{N}_{0}^{2}} \\ {\alpha\,l+\beta\,m=n}
} }\sigma(l)\sigma(m).
 \end{equation}
Let $W_{\beta}(n)$ stands for $W_{(1,\beta)}(n)$. 
We set $W_{(\alpha,\beta)}(n)=0$ if for all $(l,m)\in\mathbb{N}^{2}$ it holds that 
$\alpha\,l+\beta\,m\neq n$.

We give the values of $\alpha\beta$ for those convolution sums 
$W_{(\alpha, \beta)}(n)$ which have so far been evaluated in 
\autoref{introduction-table-1}. 
\begin{longtable}{|r|r|r|} \hline 
\textbf{Level $\alpha\beta$}  &  \textbf{Authors}   &  \textbf{References}  \\ \hline
1  &  \mB, \jwlG, \sR\  & \cite{besge,glaisher,ramanujan} \\ \hline
2, 3, 4  & \jgH\ \& \zmO\ \& \bkS\ \&  &  ~  \\
 ~  & \ksW\   & \cite{huardetal} \\ \hline
5, 7  & \mL\ \& \ksW,  \sC\ \& \pcT\   & \cite{lemire_williams,cooper_toh} \\ \hline
6  & \sA\ \& \ksW\   & \cite{alaca_williams} \\ \hline
8, 9  & \ksW\   & \cite{williams2006, williams2005} \\ \hline
10, 11, 13, 14 &  \eR   & \cite{royer} \\ \hline
12, 16, 18, 24 &  \aA\ \& \sA\ \& \ksW\   &
\cite{alaca_alaca_williams2006,alaca_alaca_williams2007,alaca_alaca_williams2007a,alaca_alaca_williams2008}
\\ \hline
15  & \bR\ \& \bS\   & \cite{ramakrishnan_sahu} \\ \hline
10, 20  & \sC\ \& \dY\   & \cite{cooper_ye2014} \\ \hline
23  & \hhC\ \& \sC\   & \cite{chan_cooper2008} \\ \hline
25  & \exwX\ \& \xlT\ \& \oxmY   & \cite{xiaetal2014} \\ \hline
27, 32  & \sA\ \& \yK   & \cite{alaca_kesicioglu2014} \\ \hline
36  & \dY\   & \cite{ye2015} \\ \hline
14, 26, 28, 30 &  \eN  & \cite{ntienjem2015} \\ \hline
22, 44, 52  &  \eN\  & \cite{ntienjem2016b}  \\ \hline
48, 64  &  \eN\  & \cite{ntienjem2016d}  \\ \hline
\caption{Known convolution sums $W_{(\alpha, \beta)}(n)$ of level $\alpha\beta$} 
\label{introduction-table-1}
\end{longtable}

From the levels $\alpha\beta$ listed in \autoref{introduction-table-1} the 
following ones do not belong to the class of positive integers that we are 
handling in this paper: $9$, $16$, $18$, $25$, $27$, $32$, $36$, $48$ and $64$.

We evaluate the convolution sum $W_{(\alpha, \beta)}(n)$ in the case where 
\begin{equation}  \label{class1-conditions}
\alpha\beta=2^{\nu}\mho \text{ with $\nu\in\{0,1,2,3\}$, and $\mho$ odd and 
squarefree positive integer.} 
\end{equation}
The evaluation of the convolution sum for a class of natural numbers and 
especially for this class is new.   
We then apply the result for this class to determine the convolution 
sum for
$\alpha\beta=33=3\cdot 11,\,\alpha\beta=40=2^{3}\cdot 5$ and 
$\alpha\beta=56=2^{3}\cdot 7$. 
Again, these explicit convolution sums have not been evaluated as yet. 
We revisit the evaluation of the convolution sums for $\alpha\beta=10$, $11$, $12$, 
$15$, $24$. The re-evaluation of the convolution sums for $\alpha\beta=10$,  
$11$, $12$, $15$, $24$ improves the previously obtained results.  

Convolution sums are applied to establish explicit formulae for the number of 
representations of a positive integer $n$ by the octonary quadratic forms 
\begin{equation} \label{introduction-eq-1}
a\,(x_{1}^{2} +x_{2}^{2} + x_{3}^{2} + x_{4}^{2})+ b\,(x_{5}^{2} + x_{6}^{2} + 
x_{7}^{2} + x_{8}^{2}),
\end{equation}
and 
\begin{equation} \label{introduction-eq-2}
c\,(x_{1}^{2} + x_{1}x_{2} + x_{2}^{2} + x_{3}^{2} + x_{3}x_{4} + x_{4}^{2})
+ d\,(x_{5}^{2} + x_{5}x_{6}+ x_{6}^{2} + x_{7}^{2} + x_{7}x_{8}+ 
x_{8}^{2}),
\end{equation}
respectively, where $(a,b),(c, d)\in \mathbb{N}^{2}$. 

Known explicit formulae 
for the number of representations of $n$ by the octonary quadratic form 
\autoref{introduction-eq-1} are referenced in \autoref{introduction-table-rep2}
and that for the octonary quadratic form \autoref{introduction-eq-2} in  
\autoref{introduction-table-rep1}.

\begin{longtable}{|r|r|r|} \hline 
$\mathbf{(a,b)}$  &  \textbf{Authors}   &  \textbf{References}  \\ \hline
(1,2)  & \ksW\   & \cite{williams2006} \\ \hline
(1,4)  & \aA\ \& \sA\ \& \ksW\   & \cite{alaca_alaca_williams2007} \\ \hline
(1,5)  & \sC\ \& \dY\   & \cite{cooper_ye2014} \\ \hline
(1,6)  & \bR\ \& \bS\   & \cite{ramakrishnan_sahu}  \\ \hline
(1,7) & \eN\  & \cite{ntienjem2015} \\ \hline 
(1,8)  & \sA\ \& \yK\   & \cite{alaca_kesicioglu2014} \\ \hline
(1,11),(1,13)   &  \eN\  & \cite{ntienjem2016b}  \\ \hline
(1,12),(1,16),  &  ~  & ~ \\ 
(3,4)   &  \eN\  & \cite{ntienjem2016d}  \\ \hline
\caption{Known representations of $n$ by the form \autoref{introduction-eq-1}
} 
\label{introduction-table-rep2}
\end{longtable}

\begin{longtable}{|r|r|r|} \hline 
$\mathbf{(c,d)}$  &  \textbf{Authors}   &  \textbf{References}  \\ \hline
(1,1)  &  \gaL\  & \cite{lomadze} \\ \hline
(1,2)  & \sA\ \& \ksW\   & \cite{alaca_williams}    \\ \hline
(1,3)  & \ksW\   & \cite{williams2005}    \\ \hline 
(1,4),(1,6),  & ~ & ~ \\
(1,8),(2,3)  & \aA\ \& \sA\ \& \ksW\   & \cite{alaca_alaca_williams2006,alaca_alaca_williams2007,alaca_alaca_williams2007a} \\ \hline 
(1,5)  & \bR\ \& \bS\   & \cite{ramakrishnan_sahu}  \\ \hline
(1,9)  & \sA\ \& \yK\   & \cite{alaca_kesicioglu2014} \\ \hline
(1,10), (2,5) & \eN\  & \cite{ntienjem2015} \\ \hline 
(1,12),(3,4)  & \dY\   & \cite{ye2015} \\ \hline
(1,16)   &  \eN\  & \cite{ntienjem2016d}  \\ \hline
\caption{Known representations of $n$ by the form \autoref{introduction-eq-2}
} 
\label{introduction-table-rep1}
\end{longtable}

Based on the structure of $\alpha$ and $\beta$,
we provide a method to determine all pairs $(a,b)\in\mathbb{N}^{2}$ 
and $(c,d)\in\mathbb{N}^{2}$ that are neccessary for the determination 
of the formulae for the number of representations of a positive integer 
by the octonary quadratic forms \autoref{introduction-eq-1} and 
\autoref{introduction-eq-2}. 
Then we determine explicit formulae for the number of representations 
of a positive integer $n$ by the octonary quadratic forms 
\autoref{introduction-eq-1} and \autoref{introduction-eq-2}, 
whenever $\alpha\beta$ has the above form and is such that 
$\alpha\beta\equiv 0\pmod{4}$ or $\alpha\beta\equiv 0\pmod{3}$. 
As an example, we determine formulae for the number of representations of a 
positive integer $n$ by octonary quadratic forms 
\autoref{introduction-eq-1} and \autoref{introduction-eq-2} using the 
convolution sums for $\alpha\beta=33=3\cdot 11$,  
$\alpha\beta=40=2^{3}\cdot 5$ and $\alpha\beta=56=2^{3}\cdot 7$, 
respectively. 

This work is structured as follows. 
In \hyperref[modularForms]{Section \ref*{modularForms}} we discuss basic 
knowledge of modular forms,  
briefly define eta functions and convolution sums. 
The evaluation of the convolution 
sum for the above class of positive integers is discussed in 
\hyperref[convolution_alpha_beta]{Section \ref*{convolution_alpha_beta}}. 
In \hyperref[representations_a_b-c_d]{Section \ref*{representations_a_b-c_d}} 
formulae for the number of representations of a positive integer by the octonary forms 
\autoref{introduction-eq-1} and \autoref{introduction-eq-2} are determined 
for this class of positive numbers. 
Examples to illustrate our method are then given in  
\hyperref[convolution_33_40_56]{Section \ref*{convolution_33_40_56}}.  
The evaluated convolution sums for $\alpha\beta=10$, $11$, $12$, 
$15$, $24$ are revisited in \hyperref[re-evaluation]{Section 
\ref*{re-evaluation}}. We determine in 
\hyperref[representations]{Section \ref*{representations}} 
formulae for the number of representations of a positive integer $n$ for 
the illustrated examples and for $(a,b)=(1,1)$, $(1,3)$, $(2,3)$, $(1,9)$. 
We then conclude in \hyperref[conclusion]{Section \ref*{conclusion}} with 
a brief outlook.

The results of this paper are obtained using Software for symbolic scientific 
computation. This software is composed of the open source software packages 
\emph{GiNaC}, \emph{Maxima}, \emph{REDUCE}, \emph{SAGE} and the commercial 
software package \emph{MAPLE}.


\section{Basic Knowledge} 
\label{modularForms}

The upper half-plane, 
$\mathbb{H}=\{ z\in \mathbb{C}~ | ~\text{Im}(z)>0\}$, 
and $\Gamma=\text{SL}_{2}(\mathbb{R})$ the group of $2\times 2$-matrices 
\begin{math}\left(\begin{smallmatrix} a & b \\ c & d \end{smallmatrix}\right)\end{math} 
such that $a,b,c,d\in\mathbb{R}$ and $ad-bc=1$ are considered in this paper. 

Let $N\in\mathbb{N}$. Then  
\begin{eqnarray*}
\Gamma(N) & = \bigl\{~\left(\begin{smallmatrix} a & b \\ c &
  d \end{smallmatrix}\right)\in\text{SL}_{2}(\mathbb{Z})~ |
  ~\left(\begin{smallmatrix} a & b \\ c &
 d\end{smallmatrix}\right)\equiv\left(\begin{smallmatrix} 1 & 0 \\ 0 & 1 
 \end{smallmatrix}\right) \pmod{N} ~\bigr\}
\end{eqnarray*}
is a subgroup of $\Gamma$ and is known as the \emph{principal congruence subgroup of 
level N}. If a subgroup $H$ of $\Gamma$ contains $\Gamma(N)$, then that subgroup 
is called a \emph{congruence subgroup of level N}.

The following congruence subgroup of level $N$ is relevant for our purpose
 \begin{align*}
\Gamma_{0}(N) & = \bigl\{~\left(\begin{smallmatrix} a & b \\ c &
  d \end{smallmatrix}\right)\in\text{SL}_{2}(\mathbb{Z})~ | ~
   c\equiv 0 \pmod{N} ~\bigr\}.
\end{align*}
Let $N\in\mathbb{N}$, $\Gamma'\subseteq\Gamma$ be a congruence 
subgroup of level $N$,  
$k\in\mathbb{Z}, \gamma\in\text{SL}_{2}(\mathbb{Z})$ and $f^{[\gamma]_{k}} : 
\mathbb{H}\cup\mathbb{Q}\cup\{\infty\} \rightarrow 
\mathbb{C}\cup\{\infty\}$ be 
defined by 
$f^{[\gamma]_{k}}(z)=(cz+d)^{-k}f(\gamma(z))$. 
The following definition is extracted from \nK's book \cite[p.~108]{koblitz-1993}.
\begin{definition} \label{modularForms-defin-2}
Suppose that $N\in\mathbb{N}$, $k\in\mathbb{Z}$, $f$ is a meromorphic function 
on $\mathbb{H}$ and $\Gamma'\subset\Gamma$ is a congruence subgroup of 
level $N$. Let furthermore $\mathbb{N}_{0}^{-}=\{\,-n\,|\,n\in\mathbb{N}_{0}\,\}$ 
be the set of all negative and nonzero natural numbers.
\begin{enumerate}
\item[(a)] $f$ is a \emph{modular function of weight $k$} for 
$\Gamma'$ if
\begin{enumerate}
\item[(a1)] for all $\gamma\in\Gamma'$ it holds that $f^{[\gamma]_{k}}=f$.
\item[(a2)] for any $\delta\in\Gamma$ it holds that $f^{[\delta]_{k}}(z)$ 
can be expressed in the form 
$\underset{n\in\mathbb{Z}}{\sum}a_{n}e^{\frac{2\pi i z n}{N}}$, 
wherein $a_{n}\neq 0$ for finitely many $n\in\mathbb{N}_{0}^{-}$.
\end{enumerate}
\item[(b)] $f$ is a \emph{modular form of weight $k$} for 
$\Gamma'$ if
	\begin{enumerate}
	\item[(b1)] $f$ is a modular function of weight $k$ for $\Gamma'$,
	\item[(b2)] $f$ is holomorphic on $\mathbb{H}$,
	\item[(b3)] for all $\delta\in\Gamma$ and for all $n\in\mathbb{N}_{0}^{-}$
it holds that $a_{n}=0$.
	\end{enumerate}
\item[(c)] $f$ is a \emph{cusp form of weight $k$ for $\Gamma'$} if
	\begin{enumerate}
	\item[(c1)] $f$ is a modular form of weight $k$ for $\Gamma'$, 
	\item[(c2)] for all $\delta\in\Gamma$ it holds that $a_{0}=0$.
	\end{enumerate}
\end{enumerate}
\end{definition}
We denote the set of modular forms of weight $k$ for $\Gamma'$ by 
$\M_{k}(\Gamma')$, the set of cusp forms of weight $k$ for 
$\Gamma'$ by $\S_{k}(\Gamma')$ and the set of 
Eisenstein forms by $\E_{k}(\Gamma')$. 
The sets $\M_{k}(\Gamma'),\,\S_{k}(\Gamma')$ and $\E_{k}(\Gamma')$ 
are vector spaces over $\mathbb{C}$. Hence, $\M_{k}(\Gamma_{0}(N))$ is the 
space of modular forms of weight $k$ for 
$\Gamma_{0}(N)$, $\S_{k}(\Gamma_{0}(N))$ is the space of 
cusp forms of weight $k$ for $\Gamma_{0}(N)$, 
and $\E_{k}(\Gamma_{0}(N))$ is the space of Eisenstein forms. 
The decomposition of the space of modular forms as a direct sum 
of the space generated by the Eisenstein series and the space of cusp 
forms, i.e., 
$\M_{k}(\Gamma_{0}(N)) =
\E_{k}(\Gamma_{0}(N))\oplus\S_{k}(\Gamma_{0}(N))$, 
is well-known; see for example 
\waS 's book (online version) \cite[p.~81]{wstein}. 


As noted in Section 5.3 of \cite[p.~86]{wstein}, if the 
primitive Dirichlet characters are trivial and $2\leq k$ is even, then 
$E_{k}(q) = 1 - \frac{2k}{B_{k}}\,\underset{n=1}{\overset{\infty}{\sum}}\,
\sigma_{k-1}(n)\,q^{n}$, where $B_{k}$ are the Bernoulli numbers.

For the purpose of this paper we only consider trivial primitive Dirichlet 
characters and $2\leq k$ even. Theorems 5.8 and 5.9 in Section 5.3 of 
\cite[p.~86]{wstein} also hold for this special case.

\subsection{Eta Quotients}  \label{etaFunctions}

The Dedekind eta function $\eta(z)$ is defined on the upper half-plane 
$\mathbb{H}$ by
$\eta(z) = e^{\frac{2\pi i z}{24}}\underset{n=1}{\overset{\infty}{\prod}}(1 - e^{2\pi i n z})$.
When we set $q=e^{2\pi i z}$, then we have 
\begin{equation*}
\eta(z) = q^{\frac{1}{24}}\underset{n=1}{\overset{\infty}{\prod}}(1 - q^{n}) = q^{\frac{1}{24}} F(q), \qquad
\text{ where } F(q)=\overset{\infty}{\underset{n=1}{\prod}} (1 - q^{n}).
\end{equation*}

Let $j,\kappa\in\mathbb{N}$ and $e_{j}\in\mathbb{Z}$. According to 
\gK \cite[p.\ 31]{koehler} an \emph{eta product} or \emph{eta quotient}, $f(z)$,  
is a finite product of Dedekind eta functions of the form
\begin{equation}  \label{basisCuspSace-eta-quotient}
\underset{j=1}{\overset{\kappa}{\prod}}\eta(jz)^{e_{j}}.
\end{equation}
Based on this definition there exists $N\in\mathbb{N}$ such that 
$N=\text{lcm}\{j~|~1\leq j\leq\kappa\}$. We call such an $N$ the \emph{level} of 
an eta product. 
An eta product will hence be understood as 
$\underset{j | N}{\overset{}{\prod}}\eta(jz)^{e_{j}}$. 
An eta product $f(z)$ behaves like a modular form of weight $k$ on 
$\Gamma_{0}(N)$ with some multiplier system whenever 
$k=\frac{1}{2}\underset{j=1}{\overset{\kappa}{\sum}}e_{j}$.

In this paper we use eta function, eta quotient and eta product interchangeably 
as synonyms.

The eta function was systematically applied by \mN\ 
\cite{newman_1957,newman_1959} to construct modular forms for 
$\Gamma_{0}(N)$ and then to determine when a function $f(z)$ was a 
modular form for $\Gamma_{0}(N)$. That is partly explained above and leads 
to conditions (i)-(iii) in the following theorem. The order of vanishing 
of an eta function at the cusps of $\Gamma_{0}(N)$, which is condition 
(iv) or (iv$'$) in \autoref{ligozat_theorem}, was determined by 
\gL\ \cite{ligozat_1975}.

In \ljpK 's book \cite[p.~99]{kilford} and \gK 's book \cite[Cor.\ 2.3, p.~37]{koehler} 
the following theorem is proved. We will use that theorem to determine eta 
quotients, $f(z)$, which belong to $\M_{k}(\Gamma_{0}(N))$, and especially 
those eta quotients which are in $\S_{k}(\Gamma_{0}(N))$. 

\begin{theorem}[\mN\ and \gL\ ] \label{ligozat_theorem} 
Let $N\in \mathbb{N}$, $D(N)$ be the set of all positive divisors of 
$N$, $\delta\in D(N)$ and $r_{\delta}\in\mathbb{Z}$. Let furthermore  
$f(z)=\overset{}{\underset{\delta\in D(N)}{\prod}}\eta^{r_{\delta}}(\delta z)$ 
be an $\eta$-quotient. 
If the following four conditions are satisfied

\begin{tabular}{llll}
{\textbf{(i)}} & $\overset{}{\underset{\delta\in D(N)}{\sum}}\delta\,r_{\delta} 
	\,\equiv 0 \pmod{24}$, & 
{\textbf{(ii)}} &  $\overset{}{\underset{\delta\in D(N)}
{\prod}}\delta^{r_{\delta}}$ \text{ is a square in } $\mathbb{Q}$, \\ 
{\textbf{(iii)}} &  $0 < \overset{}{\underset{\delta\in D(N)}
	{\sum}}r_{\delta}\, \equiv 0 \pmod{4}$,  & 
{\textbf{(iv)}}  & $\forall d\in D(N)$ \text{ it holds } 
	$\overset{}{\underset{\delta\in D(N)}{\sum}}\frac{\gcd{(\delta,d)}^{2}}	
	{\delta}\,r_{\delta} \geq 0$,  
\end{tabular}

then $f(z)\in\M_{k}(\Gamma_{0}(N))$, where 
$k=\frac{1}{2}\overset{}{\underset{\delta\in D(N)}{\sum}}r_{\delta}$. 

Moreover, the $\eta$-quotient $f(z)$ belongs to $\S_{k}(\Gamma_{0}(N))$ 
if ${\textbf{(iv)}}$ is replaced by
 
{\textbf{(iv')}} $\forall d\in D(N)$ \text{ it holds } 
$\overset{}{\underset{\delta\in D(N)}{\sum}}\frac{\gcd{(\delta,d)}^{2}}	
{\delta} r_{\delta} > 0$.
\end{theorem}

\subsection{Convolution Sums $W_{(\alpha, \beta)}(n)$}
\label{convolutionSumsEqns}

Suppose that $\alpha,\beta\in\mathbb{N}$ are such that $\alpha\leq\beta$.  
The convolution sum, 
 $W_{(\alpha,\beta)}(n)$, is defined by \autoref{def-convolution_sum}. 

Now, suppose in addition that 
$\gcd{(\alpha,\beta)}= \delta >1$.
Therefore, there exist 
$\alpha_{1}, \beta_{1}\in\mathbb{N}$ such that 
$\gcd{(\alpha_{1},\beta_{1})}=1,\ \alpha=\delta\,\alpha_{1}$ and 
$\beta =\delta\,\beta_{1}$. Then 
\begin{equation}
W_{(\alpha,\beta)}(n) = \sum_{\substack{
{(l,k)\in\mathbb{N}_{0}^{2}} \\ {\alpha\,l+\beta\,k=n}
}} \sigma(l)\sigma(k)   
 =  \sum_{\substack{
{(l,k)\in\mathbb{N}_{0}^{2}} \\ {\delta\,\alpha_{1}\,l+\delta\,\beta_{1}\,k=n}
 }}\sigma(l)\sigma(k) 
= W_{(\alpha_{1},\beta_{1})}(\frac{n}{\delta}). 
\label{convolutionSumsEqns-gcd}
\end{equation}
Therefore, we may simply assume that $\gcd{(\alpha,\beta)}=1$ as does  
\aA\ et al.\ \cite{alaca_alaca_williams2006}. 
We apply the formula proved by \mB,\ \jwlG, and \sR\ \cite{besge,glaisher,ramanujan} to 
\autoref{convolutionSumsEqns-gcd} to deduce that 
\begin{equation} \label{convolutionSum-a-a-eqn}
\forall\,\alpha\in\mathbb{N}\quad 
W_{(\alpha,\alpha)}(n) = W_{(1,1)}(\frac{n}{\alpha}) = 
\frac{5}{12}\,\sigma_{3}(\frac{n}{\alpha}) + (\frac{1}{12} - \frac{1}{2\,\alpha}n)
\sigma(\frac{n}{\alpha}). 
\end{equation}

Let $q\in\mathbb{C}$ be such that $|q|<1$. 
The Eisenstein series $L(q)$ and $M(q)$ are defined as follows:
\begin{align}  
L(q) = E_{2}(q) = 1-24\,\sum_{n=1}^{\infty}\sigma(n)q^{n}, 
\label{evalConvolClass-eqn-3} \\
M(q) = E_{4}(q) = 1 + 240\,\sum_{n=1}^{\infty}\sigma_{3}(n)q^{n}. 
\label{evalConvolClass-eqn-4}
\end{align}
We state two relevant results for the sequel of this work. These two results 
generalize the extraction of the convolution sum to all natural numbers. 
\begin{lemma}  \label{evalConvolClass-lema-1}
Let $\alpha, \beta \in \mathbb{N}$. Then 
\begin{equation*}
( \alpha\, L(q^{\alpha}) - \beta\, L(q^{\beta}) )^{2}\in
\M_{4}(\Gamma_{0}(\alpha\beta)).
\end{equation*}
\end{lemma}
 \begin{proof} 
If $\alpha=\beta$, then trivially 
$0=(\alpha\, L(q^{\alpha}) - \alpha\,L(q^{\alpha}) )^{2}\in \M_{4}
(\Gamma_{0}(\alpha))$ and there is nothing to prove. 
Therefore, we may suppose that $\alpha\neq\beta>0$ 
in the sequel. 
We apply the result proved by W.~A.~Stein \cite[Thrms 5.8, 5.9, p.~86]{wstein} 
to deduce   
$L(q) - \alpha\, L(q^{\alpha})\in \M_{2}(\Gamma_{0}(\alpha))\subseteq 
\M_{2}(\Gamma_{0}(\alpha\beta))$ and  
$L(q) - \beta\, L(q^{\beta}) \in \M_{2}(\Gamma_{0}(\beta))\subseteq
\M_{2}(\Gamma_{0}(\alpha\beta))$. Therefore, 
$$ \alpha\, L(q^{\alpha}) - \beta\,
L(q^{\beta})= (L(q)-\beta\, L(q^{\beta}) ) - (L(q)-\alpha\, L(q^{\alpha})) 
\in \M_{2}(\Gamma_{0}(\alpha\beta))$$ and so 
$(\alpha\, L(q^{\alpha}) - \beta\, L(q^{\beta}) )^{2}\in
\M_{4}(\Gamma_{0}(\alpha\beta))$.
\end{proof}
\begin{theorem} \label{convolutionSum_a_b}
Let $\alpha,\beta\in\mathbb{N}$ be such that
$\alpha < \beta$, and $\alpha$ and $\beta$ are relatively prime. 
Then
\begin{multline}
 ( \alpha\, L(q^{\alpha}) - \beta\, L(q^{\beta}) )^{2}  = 
 (\alpha - \beta)^{2} 
    + \sum_{n=1}^{\infty}\biggl(\ 240\,\alpha^{2}\,\sigma_{3}
    (\frac{n}{\alpha}) 
    + 240\,\beta^{2}\,\sigma_{3}(\frac{n}{\beta}) \\ 
    + 48\,\alpha\,(\beta-6n)\,\sigma(\frac{n}{\alpha}) 
    + 48\,\beta\,(\alpha-6n)\,\sigma(\frac{n}{\beta}) 
    - 1152\,\alpha\beta\,W_{(\alpha,\beta)}(n)\,\biggr)q^{n}. 
    \label{evalConvolClass-eqn-11}
\end{multline}
\end{theorem}
\begin{proof} 
We first observe that 
\begin{multline}
( \alpha\, L(q^{\alpha}) - \beta\, L(q^{\beta}) )^{2}  = 
\alpha^{2}\, L^{2}(q^{\alpha}) + \beta^{2}\,
             L^{2}(q^{\beta}) - 2\,\alpha\beta\, 
             L(q^{\alpha})L(q^{\beta}).
             \label{evalConvolClass-eqn-7}
\end{multline}
J.~W.~L.~Glaisher \cite{glaisher} has proved the following identity  
\begin{equation}
L^{2}(q) = 1+\sum_{n=1}^{\infty}\biggl( 240\,\sigma_{3}(n) - 288\, n\,\sigma(n) \biggr)
q^{n} \label{evalConvolClass-eqn-5}
\end{equation}
which we apply to deduce 
\begin{equation}
L^{2}(q^{\alpha}) = 1+\sum_{n=1}^{\infty}\bigl( 240\,\sigma_{3}(\frac{n}{\alpha}) -
288\,\frac{n}{\alpha}\,\sigma(\frac{n}{\alpha})\bigr) q^{n} \label{evalConvolClass-eqn-8}
\end{equation}
and 
\begin{equation}
L^{2}(q^{\beta}) = 1+\sum_{n=1}^{\infty}\bigl( 240\,\sigma_{3}(\frac{n}{\beta}) -
288\,\frac{n}{\beta}\,\sigma(\frac{n}{\beta})\bigr)q^{n}. \label{evalConvolClass-eqn-9}
\end{equation}
Since  
\begin{multline*}
\bigl(\sum_{n=1}^{\infty}\sigma(\frac{n}{\alpha})q^{n}\bigr)\bigl(\sum_{n=1}^{\infty}\sigma(\frac{n}{\beta})q^{n}\bigr)
  =
\sum_{n=1}^{\infty}\bigl(\sum_{\alpha k + \beta l=n}\sigma(k)\sigma(l)\,\bigr)q^{n}
  = \sum_{n=1}^{\infty}W_{(\alpha,\beta)}(n)q^{n},
\end{multline*}
we conclude, when using the accordingly modified  
\autoref{evalConvolClass-eqn-3}, that 
\begin{equation}
L(q^{\alpha})L(q^{\beta}) = 1 - 24\,\sum_{n=1}^{\infty}\sigma(\frac{n}{\alpha})q^{n} -
24\,\sum_{n=1}^{\infty}\sigma(\frac{n}{\beta})q^{n} + 576\,\sum_{n=1}^{\infty}W_{(\alpha,\beta)}(n)q^{n}.\label{evalConvolClass-eqn-10}
\end{equation}
Therefore, 
\begin{multline*}
\bigl( \alpha\, L(q^{\alpha}) - \beta\, L(q^{\beta})\bigr)^{2} 
   = (\alpha - \beta)^{2} + \sum_{n=1}^{\infty}\biggl(\ 240\,\alpha^{2}\,\sigma_{3}(\frac{n}{\alpha}) 
    + 240\,\beta^{2}\,\sigma_{3}(\frac{n}{\beta}) \\ 
    +  48\,\alpha\,(\beta-6n)\,\sigma(\frac{n}{\alpha})  
       + 48\,\beta\,(\alpha-6n)\,\sigma(\frac{n}{\beta}) 
     - 1152\,\alpha\beta\, W_{(\alpha,\beta)}(n)\ \biggr)q^{n}  
\end{multline*}
as asserted. 
\end{proof}


\section{Evaluating  $W_{(\alpha,\beta)}(n)$ for a class of natural numbers 
$\alpha\beta$}
\label{convolution_alpha_beta}

Suppose that $\alpha$ and $\beta$ are positive integers which satisfy the
following two conditions: 
\begin{enumerate}
  \item[(i)] $\gcd{(\alpha,\beta)}=1$ 
 \item[(ii)] \autoref{class1-conditions}.
\end{enumerate}
We derive the formula for the convolution sum $W_{(\alpha,\beta)}(n)$ for all 
such $\alpha$ and $\beta$. Let in the sequel 
$D(\alpha\beta)$ denote the set of all positive divisors of $\alpha\beta$.

\subsection{Bases for $\E_{4}(\Gamma_{0}(\alpha\beta))$ and $\S_{4}(\Gamma_{0}(\alpha\beta))$}  \label{convolution_alpha_beta-bases}

The existence of a basis of the space of cusp forms of weight $2\leq k$ even for 
$\Gamma_{0}(\alpha\beta)$ when $\alpha\beta$ 
is not a perfect square is discussed by \aP\ \cite{pizer1976}. 
We recall that the Dirichlet character $\chi$ is assumed to be trivial, that 
is $\chi=1$. 

According to the dimension formulae in 
\tM 's book \cite[Thrm 2.5.2,~p.~60]{miyake1989} or 
\cite[Prop.~6.1, p.~91]{wstein}, 
\begin{itemize}
\item and in addition due to the special form of $\alpha\beta$, we deduce that  
\begin{equation} \label{dimension-Eisenstein}
\text{dim}(\E_{4}(\Gamma_{0}(\alpha\beta)))=\underset{\delta|\alpha\beta}{\sum}\varphi
(\gcd (\delta,\frac{\alpha\beta}{\delta}))=\underset{\delta|\alpha\beta}{\sum}1 = \sigma_{0}(\alpha\beta)=d(\alpha\beta),
\end{equation} 
where $\varphi$ is the Euler's totient function.
\item we may assume that 
$\text{dim}(\S_{4}(\Gamma_{0}(\alpha\beta)))=m_{S}\in\mathbb{N}$. 
\end{itemize}
To determine as many elements of $\S_{4}(\Gamma_{0}(\alpha\beta))$ as possible for an 
explicitly given $\alpha\beta$, we use an exhaustive search when we apply  \autoref{ligozat_theorem}. 
We select from these determined elements of the space 
$\S_{4}(\Gamma_{0}(\alpha\beta))$ relevant ones for the 
purpose of the determination of a basis of this space. 

The so-determined basis of the vector space of cusp forms is not unique. 
However, due to the change of basis which is an automorphism, it is sufficient 
to only consider this basis for our purpose. 

\begin{theorem} \label{basisCusp_a_b}
\begin{enumerate}
\item[\textbf{(a)}] The set $\EuScript{B}_{E}=\{ M(q^{t})\,\mid\, t\in D(\alpha\beta)\,\}$ is a basis of $\E_{4}(\Gamma_{0}(\alpha\beta))$.
\item[\textbf{(b)}] Let $1\leq i\leq m_{S}$ be positive integers, 
$\delta\in D(\alpha\beta)$ and $(r(i,\delta))_{i,\delta}$ be a 
table of the powers of\  $\eta(\delta z)$. Let furthermore  
$\EuFrak{B}_{\alpha\beta,i}(q)=\underset{\delta|\alpha\beta}{\prod}\eta^{r(i,\delta)}(\delta
z)$ 
be selected elements of $\S_{4}(\Gamma_{0}(\alpha\beta))$. Then 
the set $\EuScript{B}_{S}=\{\, \EuFrak{B}_{\alpha\beta,i}(q)\,\mid\,1\leq i\leq m_{S}\,\}$ 
is a basis of $\S_{4}(\Gamma_{0}(\alpha\beta))$.
\item[\textbf{(c)}] The set 
$\EuScript{B}_{M}=\EuScript{B}_{E}\cup\EuScript{B}_{S}$  
constitutes a basis of $\M_{4}(\Gamma_{0}(\alpha\beta))$. 
\end{enumerate}
\end{theorem}

\begin{remark} \label{basis-remark}
\begin{enumerate}
\item[\textbf{(r1)}] For each $1\leq i\leq m_{S}$ the eta quotient 
$\EuFrak{B}_{\alpha\beta,i}(q)$ can be expressed in the form 
$\underset{n=1}{\overset{\infty}{\sum}}\EuFrak{b}_{\alpha\beta,i}(n)q^{n}$, where 
 for each $n\geq 1$ the $\EuFrak{b}_{\alpha\beta,i}(n)$ are integers. 
\item[\textbf{(r2)}] 
If we divide the summation that results from \autoref{ligozat_theorem} $(iv')$ 
when we have set $d=N$ by $24$, then we obtain a positive integer which is 
the smallest degree of $q$ in $\EuFrak{B}_{\alpha\beta,i}(q)$.
\item[\textbf{(r3)}] The proof of  \autoref{basisCusp_a_b} (b) 
provides a method to effectively determine a basis of the space of cusp forms, especially 
when the level $\alpha\beta$ is large.  
\end{enumerate}
\end{remark}

\begin{proof} 
\begin{enumerate}
\item[(a)] By Theorem 5.8 in Section 5.3 of \waS\ \cite[p.~86]{wstein}
each $M(q^{t})$ is in $\M_{4}(\Gamma_{0}(t))$, where $t\in D(\alpha\beta)$. 
Since $\E_{4}(\Gamma_{0}(\alpha\beta))$ has a finite dimension, it suffices 
to show that $M(q^{t})$ with $t\in D(\alpha\beta)$ are linearly independent. 
Suppose that $x_{t}\in\mathbb{C}$ with $t\in D(\alpha\beta)$. 


We prove this by induction on the elements of the set $D(\alpha\beta)$ which is assumed to be linearly ordered. 

The case $t=1\in D(\alpha\beta)$ is obvious since comparing the coefficients of 
$q^{t}$ on both sides of the equation $x_{t}\,M(q^{t})=0$ clearly gives $x_{t}=0$.

Suppose now that the cardinality of the set $D(\alpha\beta)$ is greater than $1$ and that $M(q^{t})$ 
are linearly independent for all $t\in D(\alpha\beta)$ such that $t\leq t_{1}$ for a given $t_{1}$ with $1<t_{1}<\alpha\beta$. Let $C$ be 
the proper non-empty subset of $D(\alpha\beta)$ which contains all positive divisors of $\alpha\beta$ less than or equal 
to $t_{1}$.  Note that all positive divisors of $t_{1}$ constitute a subset of $C$ and observe that each $t\in D(\alpha\beta)$ belongs to the class of positive integers defined by \autoref{class1-conditions}. Let us consider 
the non-empty subset $C\cup\{t'\}$ of $D(\alpha\beta)$, wherein $t'$ is the next ascendant element of $D(\alpha\beta)$ 
which is greater than $t_{1}$ the greatest element of the set $C$. Then 
\begin{equation*}
	\underset{t\in C\cup\{t'\}}{\sum}x_{t}\,M(q^{t})= \underset{t\in C}{\sum}x_{t}\,M(q^{t})+x_{t'}\,M(q^{t'})= 0.
\end{equation*}
By the induction hypothesis it holds that $x_{t}=0$ for all $t\in C$. So, we obtain from the above equation 
that $x_{t'}=0$ when we compare the coefficient of $q^{t'}$ on both sides of the equation. 

Hence, the solution of the homogeneous system of $d(\alpha\beta)$ linear equations 
is $x_{t}=0$ for all $t\in D(\alpha\beta)$. Therefore, the set 
$\EuScript{B}_{E}$ is linearly independent and hence 
is a basis of $\E_{4}(\Gamma_{0}(\alpha\beta))$.

\item[(b)]  
Since $\EuFrak{B}_{\alpha\beta,i}(q)$ with $1\leq i\leq m_{S}$ are obtained 
from an exhaustive search using \autoref{ligozat_theorem} $(i)-(iv')$, 
it holds that each $\EuFrak{B}_{\alpha\beta,i}(q)$ 
is in the space $\S_{4}(\Gamma_{0}(\alpha\beta))$.

Since the dimension of $\S_{4}(\Gamma_{0}(\alpha\beta))$ is 
$m_{S}\in\mathbb{N}$, it is
sufficient to show that the set $\{ \EuFrak{B}_{\alpha\beta,i}(q)\mid 1\leq i\leq m_{S}\}$ 
is linearly independent. 
Suppose that $x_{i}\in\mathbb{C}$ and 
$\underset{i=1}{\overset{m_{S}}{\sum}}x_{i}\,\EuFrak{B}_{\alpha\beta,i}(q)=0$. Then   
\begin{equation*}
\underset{i=1}{\overset{m_{S}}{\sum}}x_{i}\,\EuFrak{B}_{\alpha\beta,i}(q)
= \underset{n=1}{\overset{\infty}{\sum}}(\,\underset{i=1}{\overset{m_{S}}{\sum}}x_{i}\,\EuFrak{b}_{\alpha\beta,i}(n)\,)q^{n} = 0
\end{equation*}
which gives the following homogeneous system of $m_{S}$ linear equations  
in $m_{S}$ unknowns
\begin{equation}  \label{basis-cusp-eqn-sol}
\underset{i=1}{\overset{m_{S}}{\sum}}\EuFrak{b}_{\alpha\beta,i}(n)\,x_{i}= 0,\qquad 
1\leq n\leq m_{S}.
\end{equation}
By \hyperref[basis-remark]{Remark \ref*{basis-remark}} \textbf{(r2)} we may consider 
without lost of generality two cases.
\begin{description}
  \item[Case 1] For each $1\leq i\leq m_{S}$ the smallest degree of $q$ 
      in $\EuFrak{B}_{\alpha\beta,i}(q)$ is $i$. 
      It is 
      then obvious that the $m_{S}\times m_{S}$ matrix which corresponds to 
      this homogeneous system of linear equations is triangular with $1$'s on the 
      diagonal. Hence, the determinant of that matrix is $1$ and so 
      $x_{i}=0$ for all $1\leq i\leq m_{S}$. 
  \item[Case 2] The set $\EuScript{B}_{S}$ does contain a subset, say 
    $\EuScript{B}_{S}'=\{\,\EuFrak{B}_{\alpha\beta,i}(q)\,\mid\,1\leq i\leq u\,\}$ for 
    some $1\leq u< m_{s}$, for which the smallest degree of $q$ 
    in $\EuFrak{B}_{\alpha\beta,i}(q)$ is $i$. Since the smallest degree of 
    $q$ is $1$, 
    $\EuScript{B}_{S}'$ is not empty. 
    Let us consider $\EuScript{B}_{S}$ as an 
    ordered set of the form $\EuScript{B}_{S}'\cup\EuScript{B}_{S}''$, where 
    $\EuScript{B}_{S}''=\{\,\EuFrak{B}_{\alpha\beta,i}(q)\,\mid\,u< i\leq
    m_{S}\,\}$. Let $A=(\EuFrak{b}_{\alpha\beta,i}(n))$ be the 
    $m_{S}\times m_{S}$ matrix 
    in \hyperref[basis-cusp-eqn-sol]{Equation \ref*{basis-cusp-eqn-sol}}. 
    In this matrix $i$ indicates the $i$-th column and $n$ indicates 
    the $n$-th row. Note that applying case 1 the subset $\EuScript{B}_{S}'$ is 
    linearly independent since the determinant of the corresponding $u\times u$ 
    matrix is $1$.
    
    If $\text{det}(A)\neq 0$, then $x_{i}=0$ for all $1\leq i\leq m_{S}$. 
    Suppose now that $\text{det}(A)= 0$. Then for 
    some $u< k\leq m_{S}$ there exists $\EuFrak{B}_{\alpha\beta,k}(q)$ which is causing 
    the system of linear equations to be 
    inconsistent. We substitute $\EuFrak{B}_{\alpha\beta,k}(q)$ with, say 
    $\EuFrak{B}_{\alpha\beta,k}'(q)$, which does not occur in $\EuScript{B}_{S}$ and 
    compute the determinant of the new matrix $A$. 
    Since there are finitely many 
    $\EuFrak{B}_{\alpha\beta,k}(q)$ with $u< k\leq m_{S}$ that may cause the 
    system of linear equations to be inconsistent and finitely many elements of 
    $\S_{4}(\Gamma_{0}(\alpha\beta))\setminus \EuScript{B}_{S}$, 
    the procedure will terminate with a consistent system of linear equations. 
    So, $x_{i}=0$ for all $1\leq i\leq m_{S}$.
\end{description}
Therefore, the set $\{ \EuFrak{B}_{\alpha\beta,i}(q)\mid 1\leq i\leq m_{S}\}$ 
is linearly independent and so is a basis of $\S_{4}(\Gamma_{0}(\alpha\beta))$.

\item[(c)] Since $\M_{4}(\Gamma_{0}(\alpha\beta))=\E_{4}(\Gamma_{0}(\alpha\beta))\oplus 
\S_{4}(\Gamma_{0}(\alpha\beta))$, the result follows from (a) and (b).
\end{enumerate}
\end{proof}

The proof of \autoref{basisCusp_a_b}(b) provides an effective method to 
determine the basis of the space of cusp forms of level $\alpha\beta$ whenever 
$\alpha\beta$ belongs to this class of positive integers.

\subsection{Evaluating the convolution sum $W_{(\alpha,\beta)}(n)$}
\label{convolution_alpha_beta-gen}

\begin{lemma} \label{convolution-lemma_a_b}
Let $\alpha,\beta\in\mathbb{N}$ be such that $\gcd{(\alpha,\beta)}=1$. 
Let furthermore $\EuScript{B}_{M}=\EuScript{B}_{E}\cup\EuScript{B}_{S}$ be a 
basis of $\M(\Gamma_{0}(\alpha\beta))$. Then there exist
$X_{\delta}\in\mathbb{C}$ and $Y_{j}\in\mathbb{C}$ with $\delta\in D(\alpha\beta)$
and $1\leq j\leq m_{S}$ such that 
\begin{multline}
(\alpha\, L(q^{\alpha}) - \beta\, L(q^{\beta}))^{2} 
 = \sum_{\delta|\alpha\beta}X_{\delta} + \sum_{n=1}^{\infty}\biggl(\, 
 240\,\sum_{\delta|\alpha\beta}\,\sigma_{3}(\frac{n}{\delta})\,X_{\delta} 
 + \sum_{j=1}^{m_{S}}\,\EuFrak{b}_{\alpha\beta,j}(n)\,Y_{j}\, \biggr)q^{n}.
\label{convolution_a_b-eqn-0}
\end{multline}
\end{lemma}
\begin{proof} That $( \alpha L(q^{\alpha}) - \beta L(q^{\beta}) )^{2}\in
\M_{4}(\Gamma_{0}(\alpha\beta))$ follows from 
\hyperref[evalConvolClass-lema-1]{Lemma \ref*{evalConvolClass-lema-1}}.  
Hence, by \autoref{basisCusp_a_b}\,(c), there exist 
$X_{\delta},Y_{j}\in\mathbb{C}, 
1\leq j\leq m_{S}\text{ and } \delta\in D(\alpha\beta)$, such that  
\begin{align}
( \alpha L(q^{\alpha}) - \beta L(q^{\beta}) )^{2}    &
= \sum_{\delta|\alpha\beta}X_{\delta}\,M(q^{\delta}) + \sum_{j=1}^{m_{S}}\,Y_{j}\ \EuFrak{B}_{\alpha\beta,j}(q)   \notag  \\ &
= \sum_{\delta|\alpha\beta}X_{\delta} + \sum_{n=1}^{\infty}\biggl(\, 
 240\,\sum_{\delta|\alpha\beta}\,\sigma_{3}(\frac{n}{\delta})\,X_{\delta} + 
 \sum_{j=1}^{m_{S}}\,\EuFrak{b}_{\alpha\beta,j}(n)\,Y_{j}\, \biggr)q^{n}. 
\label{convolution_a_b-eqn-0a}
\end{align} 
We compare the coefficients of $q^{n}$ on the right hand side of 
\autoref{convolution_a_b-eqn-0a} with that on the right hand side of   
\autoref{evalConvolClass-eqn-11}, 
which yields 
\begin{multline*}
\sum_{n=1}^{\infty}\biggl(\, 
 240\,\sum_{\delta|\alpha\beta}\,\sigma_{3}(\frac{n}{\delta})\,X_{\delta} + 
 \sum_{j=1}^{m_{S}}\,\EuFrak{b}_{\alpha\beta,j}(n)\,Y_{j}\, \biggr)q^{n} = \sum_{n=1}^{\infty}\biggl(\ 240\,\alpha^{2}\,\sigma_{3}
    (\frac{n}{\alpha}) 
    + 240\,\beta^{2}\,\sigma_{3}(\frac{n}{\beta}) \\
    + 48\,\alpha\,(\beta-6\,n)\,\sigma(\frac{n}{\alpha})  
    + 48\,\beta\,(\alpha-6\,n)\,\sigma(\frac{n}{\beta})   
    - 1152\,\alpha\,\beta\,W_{(\alpha,\beta)}(n)\,\biggr)q^{n}. 
\end{multline*}
We then take the coefficients of $q^{n}$ for which $n$ is in $D(\alpha\beta)$
and $1\leq n\leq m_{S}$, but as many as the unknowns $X_{\delta}$ and $Y_{j}$.
This results in  
a system of $d(\alpha\beta) + m_{S}$ linear equations whose unique solution determines the values of 
the unknown $X_{\delta}$ for all $\delta\in D(\alpha\beta)$ and the values 
of the unknown $Y_{j}$ for all $1\leq j\leq m_{S}$. 
Hence, we obtain the stated result.
\end{proof}
In the following theorem, let $X_{\delta}$ and $Y_{j}$ stand for their values 
obtained in the previous theorem.
\begin{theorem} \label{convolution_a_b}
Let $n$ be a positive integer. Then 
\begin{align*}
 W_{(\alpha,\beta)}(n)  =  &  
 -\frac{5}{24\,\alpha\,\beta}\,\sum_{\substack{{\delta|
 \alpha\beta}\\{\delta\neq\alpha,\beta}}}X_{\delta}\,\sigma_{3}(\frac{n}{\delta}) 
 + \frac{5}{24\,\alpha\,\beta}\,(\alpha^{2} - X_{\alpha})\,\sigma_{3}(\frac{n}{\alpha})  \\ &
 + \frac{5}{24\,\alpha\,\beta}\,(\beta^{2} - X_{\beta})\,\sigma_{3}(\frac{n}{\beta})  
 - \sum_{j=1}^{m_{S}}\,\frac{1}{1152\,\alpha\,\beta}\,Y_{j}\,\EuFrak{b}_{\alpha\beta,j}(n) \\ &
 + (\frac{1}{24}-\frac{1}{4\beta}n)\sigma(\frac{n}{\alpha}) 
 + (\frac{1}{24}-\frac{1}{4\alpha}n)\sigma(\frac{n}{\beta}).  
\end{align*}
\end{theorem}
\begin{proof}  We set the right hand side of  
\autoref{convolution_a_b-eqn-0} with that of 
\autoref{evalConvolClass-eqn-11} equal, 
which yields   
 \begin{multline*} 
 1152\,\alpha\,\beta\,W_{(\alpha,\beta)}(n) = 
 - 240\,\sum_{\delta|\alpha\beta}\,\sigma_{3}(\frac{n}{\delta})\,X_{\delta} 
 - \sum_{j=1}^{m_{S}}\,\EuFrak{b}_{\alpha\beta,j}(n)\,Y_{j} + 240\,\alpha^{2}\,\sigma_{3} 
    (\frac{n}{\alpha}) \\
    + 240\,\beta^{2}\,\sigma_{3}(\frac{n}{\beta}) 
    + 48\,\alpha\,(\beta-6\,n)\,\sigma(\frac{n}{\alpha})  
    + 48\,\beta\,(\alpha-6\,n)\,\sigma(\frac{n}{\beta}).   
\end{multline*}
Then we solve for $W_{(\alpha,\beta)}(n)$ to obtain the stated result.
\end{proof}
\begin{remark}
Observe that the following part of \autoref{convolution_a_b}  
$$(\frac{1}{24}-\frac{1}{4\beta}n)\sigma(\frac{n}{\alpha}) 
 + (\frac{1}{24}-\frac{1}{4\alpha}n)\sigma(\frac{n}{\beta})$$
depends only on $n,\alpha$ and $\beta$ and not on the basis of the 
modular space $\M_{4}(\Gamma_{0}(\alpha\beta))$. 
\end{remark}


\section{Number of Representations of a positive Integer for this Class 
of positive Integer}
\label{representations_a_b-c_d}

We discuss in this section the determination of formulae for the number of 
representations of a positive integer by the octonary quadratic forms 
\autoref{introduction-eq-1} and \autoref{introduction-eq-2}, respectively.

\subsection{Representations of a positive Integer by the Octonary Quadratic Form  \autoref{introduction-eq-1}
}
\label{representations_a_b}
We restrict the general form of $\alpha\beta$ to   
$2^{\nu}\mho$
where $\nu\in\{2, 3\}$ and $\mho$ is odd squarefree finite product of distinct 
odd primes; that is $\alpha\beta\equiv 0\pmod{4}$. 

\subsubsection{Determination of $(a,b)\in\mathbb{N}^{2}$} 
\label{determine_a_b}
We carry out a method to determine all pairs $(a,b)\in\mathbb{N}^{2}$ necessary for
the determination of $N_{(a,b)}(n)$ for a given $\alpha\beta\in\mathbb{N}$ 
which belongs to the above class. 

Let $\Lambda=\frac{\alpha\beta}{4}=2^{\nu-2}\mho$, 
$P_{4}=\{p_{0}=2^{\nu-2}\}\cup\underset{j>1}{\bigcup}\{ p_{j}\,\mid\, p_{j}
\text{ is a prime divisor of } \mho\,\}$ and  
$\EuScript{P}(P_{4})$ be the power set of $P_{4}$. Then for each $Q\in\EuScript{P}(P_{4})$ 
we define $\mu(Q)=\underset{p\in Q}{\prod}p$. We set $\mu(Q)=1$ if $Q$ is an 
empty set. Let now 
\begin{multline*}
\Omega_{4}=\{(\mu(Q_{1}),\mu(Q_{2}))~|~\text{ there exist } Q_{1},Q_{2}\in\EuScript{P}(P_{4}) 
\text{ such that }  \\ 
\gcd{(\mu(Q_{1}),\mu(Q_{2}))}=1\,\text{ and } 
\mu(Q_{1})\,\mu(Q_{2})=\Lambda\,\}.
\end{multline*}
Observe that $\Omega_{4}\neq\emptyset$ since $(1,\Lambda)\in\Omega_{4}$.

To illustrate our method, suppose that $\alpha\beta=2^{3}\cdot 3\cdot 5$. Then 
$\Lambda=2\cdot 3\cdot 5$, $P_{4}=\{2,3,5\}$ and 
$\Omega_{4}=\{(1,30),(2,15),(3,10),(5,6)\}$.
\begin{proposition} \label{representation-prop-1}
Suppose that $\alpha\beta$ has the above restricted form and suppose that $\Omega_{4}$ is 
defined as above. Then for all $n\in\mathbb{N}$ the set $\Omega_{4}$ 
contains all pairs $(a,b)\in\mathbb{N}^{2}$ such that $N_{(a,b)}(n)$ can 
be obtained by applying $W_{(\alpha,\beta)}(n)$ and some other evaluated 
convolution sums.
\end{proposition}
\begin{proof}
We prove this by induction on the structure of $\alpha\,\beta$.

Suppose that $\alpha\beta=2^{\nu}p_{2}$, where $\nu\in\{2, 3\}$ and $p_{2}$ is 
an odd prime. Then by the above definitions we have $\Lambda=2^{\nu-2}p_{2}$, 
$P_{4}=\{\,2^{\nu-2},p_{2}\,\}$, 
$$\EuScript{P}(P_{4})=\{\,\emptyset,\{2^{\nu-2}\},\{p_{2}\},\{\,2^{\nu-2},p_{2}\}\,\},$$
and $\Omega_{4}=\{\,(1,2^{\nu-2}p_{2}),(2^{\nu-2},p_{2}) \,\}$.
 
We show that $\Omega_{4}$ is the largest such set. 
Assume now that there exist another set, 
say $\Omega'_{4}$, which results from the
above definitions. Then there are two cases.
\begin{description} 
\item[Case $\Omega'_{4}\subseteq\Omega_{4}$] There is nothing to show. 
So, we are done.
\item[Case $\Omega_{4}\subset\Omega'_{4}$] Let
$(e,f)\in\Omega_{4}'\setminus\Omega_{4}$.  Since $ef=2^{\nu-2}p_{2}$ and 
$\gcd{(e,f)}=1$, we must
have either $(e,f)=(1,2^{\nu-2}p_{2})$ or $(e,f)=(2^{\nu-2},p_{2})$. So, 
$(e,f)\in\Omega_{4}$. Hence, $\Omega_{4}=\Omega'_{4}$.
\end{description} 
Suppose now that $\alpha\beta=2^{\nu}p_{2}p_{3}$, where $\nu\in\{2, 3\}$ and 
$p_{2},p_{3}$ are distinct odd primes. Then by the induction hypothesis and by
the above definitions we have essentially  
$$\Omega_{4}=\{\,(1,2^{\nu-2}p_{2}p_{3}),(2^{\nu-2},p_{2}p_{3}),(2^{\nu-2}p_{2},p_{3}),(2^{\nu-2}p_{3},p_{2})
\,\}.$$ 
Again, we show that $\Omega_{4}$ is the largest such set. 
Suppose that there exist another set, 
say $\Omega'_{4}$, which 
results from the above definitions. Two cases arise.
\begin{description} 
\item[Case $\Omega'_{4}\subseteq\Omega_{4}$] There is nothing to prove. 
So, we are done.
\item[Case $\Omega_{4}\subset\Omega'_{4}$] Let
$(e,f)\in\Omega_{4}'\setminus\Omega_{4}$.  Since $ef=2^{\nu-2}p_{2}p_{3}$ and
$\gcd{(e,f)}=1$, we must
have $(e,f)=(1,2^{\nu-2}p_{2}p_{3})$ or $(e,f)=(2^{\nu-2},p_{2}p_{3})$ or 
$(e,f)=(2^{\nu-2}p_{2},p_{3})$ or $(e,f)=(2^{\nu-2}p_{3},p_{2})$. 
So, $(e,f)\in\Omega_{4}$. Hence, $\Omega_{4}=\Omega'_{4}$.
\end{description} 
\end{proof}

\subsubsection{Formulae for the Number of Representations by 
\autoref{introduction-eq-1}
 }
\label{represent_a_b}
As an immediate application of \autoref{convolution_a_b} a formula for the number 
of representations of a positive integer $n$  by the octonary quadratic form 
\autoref{introduction-eq-1} is determined for each $(a,b)\in\Omega_{4}$.

Let $n\in\mathbb{N}_{0}$ and let the number of representations
of $n$ by the qua\-tern\-ary qua\-drat\-ic form  $x_{1}^{2} +x_{2}^{2}+x_{3}^{2} +
x_{4}^{2}$ be  
$r_{4}(n)=\text{card}(\{(x_{1},x_{2},x_{3},x_{4})\in\mathbb{Z}^{4}~|~ n = x_{1}^{2} +x_{2}^{2} + x_{3}^{2} + x_{4}^{2}\})$.
It follows from the definition that $r_{4}(0) = 1$. For all $n\in\mathbb{N}$, 
the following Jacobi's identity is proved in \ksW ' book 
\cite[Thrm 9.5, p.\ 83]{williams2011} 
\begin{equation}
r_{4}(n) = 8\sigma(n) - 32\sigma(\frac{n}{4}). \label{representations-eqn-4-1}
\end{equation}

Now, let the number of representations of $n$ by the octonary quadratic form 
\autoref{introduction-eq-1} be 
\begin{multline*}
N_{(a,b)}(n) 
=\text{card}
(\{(x_{1},x_{2},x_{3},x_{4},x_{5},x_{6},x_{7},x_{8})\in\mathbb{Z}^{8}~|~
n = a\,( x_{1}^{2} +x_{2}^{2}  
    + x_{3}^{2} + x_{4}^{2} ) \\ + 
b\,( x_{5}^{2} +x_{6}^{2} + x_{7}^{2} + x_{8}^{2}) \}),
\end{multline*}
where $a,b\in\mathbb{N}$. 
Let $1<\lambda\in\mathbb{N}$ and $\tau:\mathbb{N}\mapsto\mathbb{N}$ be an 
injective function such that $\tau(n)=\lambda\cdot n$ for each $n\in\mathbb{N}$.

We then derive the following result:
\begin{theorem} \label{representations-theor_a_b}
Let $n\in\mathbb{N}$ and let $(a,b)\in\Omega_{4}$. Then  
\begin{align*}
N_{(a,b)}(n)  = ~ & 
8\sigma(\frac{n}{a}) - 32\sigma(\frac{n}{4a}) 
+ 8\sigma(\frac{n}{b}) - 32\sigma(\frac{n}{4b}) + 64\, W_{(a,b)}(n) 
+ 1024\, W_{(a,b)}(\frac{n}{4})   \\ & 
- 256\, \biggl( W_{(4a,b)}(n) + W_{(a,4b)}(n) \biggr). 
\end{align*}
\end{theorem}
\begin{proof} 
We have 
\begin{multline*}
N_{(a,b)}(n)  = \sum_{\substack{
{(l,m)\in\mathbb{N}_{0}^{2}} \\ {a\,l+b\,m=n}
 }}r_{4}(l)r_{4}(m) 
   = r_{4}(\frac{n}{a})r_{4}(0) + r_{4}(0)r_{4}(\frac{n}{b}) 
   + \sum_{\substack{
{(l,m)\in\mathbb{N}^{2}} \\ {a\,l+b\,m=n}
 }}r_{4}(l)r_{4}(m).
\end{multline*}
We make use of \autoref{representations-eqn-4-1} to obtain 
\begin{multline*}
N_{(a,b)}(n)  = 8\sigma(\frac{n}{a}) - 32\sigma(\frac{n}{4a}) + 8\sigma(\frac{n}{b}) -
32\sigma(\frac{n}{4b}) \\
   + \sum_{\substack{
{(l,m)\in\mathbb{N}^{2}} \\ {al+bm=n}
  }} (8\sigma(l) - 32\sigma(\frac{l}{4}))(8\sigma(m) - 32\sigma(\frac{m}{4})). 
\end{multline*}
We know that 
\begin{multline*}
(8\sigma(l) - 32\sigma(\frac{l}{4}))(8\sigma(m) - 32\sigma(\frac{m}{4}))  = 
64\sigma(l)\sigma(m) - 256\sigma(\frac{l}{4})\sigma(m) \\
   - 256\sigma(l)\sigma(\frac{m}{4})  + 1024\sigma(\frac{l}
   {4})\sigma(\frac{m}{4}).
\end{multline*}
In the sequel of this proof, we assume that the evaluation of 
\begin{equation*}
W_{(a,b)}(n) = \sum_{\substack{
{(l,m)\in\mathbb{N}^{2}} \\ {al+bm=n}
 }}\sigma(l)\sigma(m),
\end{equation*}
$W_{(4a,b)}(n)$ and $W_{(a,4b)}(n)$ are known. 
We set $\lambda = 4$ in the sequel. When we use the function 
$\tau$ with $l$ as argument we derive  
\begin{equation*}
W_{(4a,b)}(n) = \sum_{\substack{
{(l,m)\in\mathbb{N}^{2}} \\ {al+bm=n}
}}\sigma(\frac{l}{4})\sigma(m) 
 = \sum_{\substack{
{(l,m)\in\mathbb{N}^{2}} \\ {4a\,l+bm=n}
 }}\sigma(l)\sigma(m).
\end{equation*}
When we apply 
the function $\tau$ with $m$ as argument we infer 
\begin{equation*}
W_{(a,4b)}(n) = \sum_{\substack{
{(l,m)\in\mathbb{N}^{2}} \\ {al+bm=n}
}}\sigma(l)\sigma(\frac{m}{4})  = \sum_{\substack{
{(l,m)\in\mathbb{N}^{2}} \\ {al+4b\,m=n}
 }}\sigma(l)\sigma(m).
\end{equation*}
We simultaneously apply the function $\tau$ with $l$ and $m$ as arguments, 
respectively, to conclude
\begin{equation*}
\sum_{\substack{
{(l,m)\in\mathbb{N}^{2}} \\ {al+bm=n}
}}\sigma(\frac{l}{4})\sigma(\frac{m}{4}) 
= \sum_{\substack{
{(l,m)\in\mathbb{N}^{2}} \\ {al+bm=\frac{n}{4}}
 }}\sigma(l)\sigma(m)
 = W_{(a,b)}(\frac{n}{4}).
\end{equation*}
 
We finally put all these evaluations together to obtain the stated result 
for $N_{(a,b)}(n)$. 
\end{proof}


\subsection{Representations of a Positive Integer by the Octonary Quadratic Form \autoref{introduction-eq-2}
} 
\label{representations_c_d}

In this case, the general form of $\alpha\beta$ is restricted to 
$2^{\nu}\mho$, where $\mho\equiv 0\pmod{3}$.

\subsubsection{Determination of $(c,d)\in\mathbb{N}^{2}$} 
\label{determine_c_d}
The following method determine all pairs $(c,d)\in\mathbb{N}^{2}$ necessary for
the determination of $R_{(c,d)}(n)$ for a given $\alpha\beta\in\mathbb{N}$ 
belonging to the above class. The following method is quasi similar to the one 
used in \hyperref[determine_a_b]{Subsection \ref*{determine_a_b}}. 

Let $\Delta=\frac{\alpha\beta}{3}=\frac{2^{\nu}\mho}{3}$. Let 
$P_{3}=\{p_{0}=2^{\nu}\}\cup\underset{j>2}{\bigcup}\{ p_{j}\,\mid\,p_{j}\text{ is
a prime divisor of } \mho\,\}$. 
Let $\EuScript{P}(P_{3})$ be the power set of $P_{3}$. Then for each 
$Q\in\EuScript{P}(P_{3})$ 
we define $\mu(Q)=\underset{p\in Q}{\prod}p$. We set $\mu(Q)=1$ if 
$Q$ is an empty set.
Let now $\Omega_{3}$ be defined in a similar way as $\Omega_{4}$ in 
\hyperref[determine_a_b]{Subsection \ref*{determine_a_b}}, 
however with $\Delta$ instead of $\Lambda$, i.e., 
\begin{multline*}
\Omega_{3}=\{(\mu(Q_{1}),\mu(Q_{2}))~|~\text{ there exist } Q_{1},Q_{2}\in\EuScript{P}(P_{3}) 
\text{ such that }  \\ 
\gcd{(\mu(Q_{1}),\mu(Q_{2}))}=1\,\text{ and } 
\mu(Q_{1})\,\mu(Q_{2})=\Delta\,\}.
\end{multline*}
Note that $\Omega_{3}\neq\emptyset$ since $(1,\Delta)\in\Omega_{3}$.
As an example, suppose again that $\alpha\beta=2^{3}\cdot 3\cdot 5$. Then 
$\Delta=2^{3}\cdot 5$, $P_{3}=\{2^{3},5\}$ and $\Omega_{3}=\{(1,40),(5,8)\}$.
\begin{proposition} \label{representation-prop-2}
Suppose that $\alpha\beta$ has the above restricted form and Suppose that $\Omega_{3}$ 
be defined as above. Then for all $n\in\mathbb{N}$ the set $\Omega_{3}$ 
contains all pairs $(c,d)\in\mathbb{N}^{2}$ such that $R_{(c,d)}(n)$ can 
be obtained by applying $W_{(\alpha,\beta)}(n)$ and some other evaluated 
convolution sums.
\end{proposition}
\begin{proof} 
Simlar to the proof of 
\hyperref[representation-prop-1]{Proposition \ref*{representation-prop-1}}.
\end{proof}

\subsubsection{Formulae for the Number of Representations by 
 \autoref{introduction-eq-2}
 }
\label{represent_c_d}
We apply \autoref{convolution_a_b} to determine a formula for the number 
of representations of a positive integer $n$  by the octonary quadratic form  
\autoref{introduction-eq-2} for each $(c,d)\in\Omega_{3}$.

Let $n\in\mathbb{N}_{0}$ and  let $s_{4}(n)$ denote the number of representations
of $n$ by the quaternary quadratic form  $x_{1}^{2} + x_{1}x_{2} + x_{2}^{2} +
x_{3}^{2} + x_{3}x_{4} + x_{4}^{2}$, that is,   
\begin{multline*}
s_{4}(n)=\text{card}(\{(x_{1},x_{2},x_{3},x_{4})\in\mathbb{Z}^{4}~|~  
n = x_{1}^{2} + x_{1}x_{2} + x_{2}^{2} + x_{3}^{2} + x_{3}x_{4} + x_{4}^{2}\}).
\end{multline*}
It is obvious that $s_{4}(0) = 1$. \jgH\ et al.\ \cite{huardetal}, \gaL\
\cite{lomadze} and \ksW\ \cite[Thrm 17.3, p.\ 225]{williams2011} have 
proved that for all $n\in\mathbb{N}$ 
\begin{equation}
s_{4}(n) = 12\sigma(n) - 36\sigma(\frac{n}{3}). \label{representations-eqn-c_d-1}
\end{equation}
Now, let the number of representations of $n$ by the octonary quadratic form 
\autoref{introduction-eq-2} be 
\begin{multline*}
R_{(c,d)}(n) =\text{card}
(\{(x_{1},x_{2},x_{3},x_{4},x_{5},x_{6},x_{7},x_{8})\in\mathbb{Z}^{8}~|~
n = c\,(x_{1}^{2} + x_{1}x_{2}  \\
   + x_{2}^{2}  + x_{3}^{2} + x_{3}x_{4} + x_{4}^{2}) + 
d\,( x_{5}^{2} + x_{5}x_{6}+ x_{6}^{2} + x_{7}^{2} + x_{7}x_{8}+ x_{8}^{2}) \}).
\end{multline*}
Let $\lambda$ and $\tau$ be defined as in 
\hyperref[represent_a_b]{Section \ref*{represent_a_b}}.

We infer the following result:
\begin{theorem} \label{representations-theor-c_d}
Let $n\in\mathbb{N}$ and $(c,d)\in\Omega_{3}$. Then  
\begin{align*}
R_{(c,d)}(n) =  & 
12\sigma(\frac{n}{c}) - 36\sigma(\frac{n}{3c}) + 12\sigma(\frac{n}{d}) -
36\sigma(\frac{n}{3d}) + 144\, W_{(c,d)}(n) 
   + 1296\, W_{(c,d)}(\frac{n}{3}) \\ & 
   - 432\, \biggl( W_{(3c,d)}(n) +  W_{(c,3d)}(n) \biggr). 
\end{align*}
\end{theorem}
\begin{proof} It holds that 
\begin{multline*}
R_{(c,d)}(n) = \sum_{\substack{
{(l,m)\in\mathbb{N}_{0}^{2}} \\ {cl+dm=n}
}}s_{4}(l)s_{4}(m) 
  = s_{4}(\frac{n}{c})s_{4}(0) + s_{4}(0)s_{4}(\frac{n}{d}) 
+ \sum_{\substack{
{(l,m)\in\mathbb{N}^{2}} \\ {cl+dm=n}
}}s_{4}(l)s_{4}(m). 
\end{multline*}
We apply \autoref{representations-eqn-c_d-1} to derive 
\begin{multline*}
R_{(c,d)}(n) = 12\sigma(\frac{n}{c}) - 36\sigma(\frac{n}{3c}) + 12\sigma(\frac{n}{d}) -
36\sigma(\frac{n}{3d}) \\
   + \sum_{\substack{
{(l,m)\in\mathbb{N}^{2}} \\ {cl+dm=n}
 }} (12\sigma(l) -
  36\sigma(\frac{l}{3}))(12\sigma(m) - 36\sigma(\frac{m}{3})).
\end{multline*}
We know that 
\begin{multline*}
(12\sigma(l) - 36\sigma(\frac{l}{3}))(12\sigma(m) - 36\sigma(\frac{m}{3})) = 
144\sigma(l)\sigma(m) - 432\sigma(\frac{l}{3})\sigma(m) \\
   - 432\sigma(l)\sigma(\frac{m}{3})  + 1296\sigma(\frac{l}{3})\sigma(\frac{m}{3}).
\end{multline*}
We assume that the evaluation of 
\begin{equation*}
 W_{(c,d)}(n) = \sum_{\substack{
{(l,m)\in\mathbb{N}^{2}} \\ {cl+dm=n}
 }}\sigma(l)\sigma(m),
\end{equation*}
$W_{(c,3d)}(n)$ and $W_{(3c,d)}(n)$ are known.  
We set $\lambda=3$ in the sequel. We apply the function $\tau$ to $m$ to derive 
\begin{equation*}
\sum_{\substack{
{(l,m)\in\mathbb{N}^{2}} \\ {cl+dm=n}
 }}\sigma(l)\sigma(\frac{m}{3})  = 
 \sum_{\substack{
{(l,m)\in\mathbb{N}^{2}} \\ {cl+3d\,m=n}
 }}\sigma(l)\sigma(m)
 = W_{(c,3d)}(n). 
\end{equation*}
We make use of the function $\tau$ with $l$ as argument to conclude 
\begin{equation*}
\sum_{\substack{
{(l,m)\in\mathbb{N}^{2}} \\ {cl+dm=n}
 }}\sigma(m)\sigma(\frac{l}{3})  = 
\sum_{\substack{
{(l,m)\in\mathbb{N}^{2}} \\ {3c\,l+dm=n}
 }}\sigma(l)\sigma(m)
= W_{(3c,d)}(n).
\end{equation*}
We simultaneously apply apply the function $\tau$ to $l$ and to $m$ as arguments, 
respectively, to infer
\begin{equation*}
\sum_{\substack{
{(l,m)\in\mathbb{N}^{2}} \\ {cl+dm=n}
 }}\sigma(\frac{m}{3})\sigma(\frac{l}{3})  = 
\sum_{\substack{
{(l,m)\in\mathbb{N}^{2}} \\ {cl+dm=\frac{n}{3}}
 }}\sigma(l)\sigma(m)
  = W_{(c,d)}(\frac{n}{3}).
\end{equation*}

Finally, we bring all these evaluations together to obtain the stated result 
for $R_{(c,d)}(n)$.
\end{proof}


\section{Evaluation of the convolution sums when $\alpha\beta = 33, 40,56$ }
\label{convolution_33_40_56}

In this section, we give explicit formulae for the convolution sum 
$W_{(\alpha,\beta)}(n)$  when 
$\alpha\beta=33=3\cdot 11,\alpha\beta=40=2^{3}\cdot 5$ and 
$\alpha\beta=56=2^{3}\cdot 7$. 

When we apply \tM\ \cite[Lma 2.1.3, p.\ 41]{miyake1989}, we conclude that 
\begin{align}
\M_{4}(\Gamma_{0}(11))\subset\M_{4}(\Gamma_{0}(33)) \label{eqn-11-33} \\ 
\M_{4}(\Gamma_{0}(5))\subset\M_{4}(\Gamma_{0}(10))\subset\M_{4}(\Gamma_{0}(20))\subset
\M_{4}(\Gamma_{0}(40))  \label{eqn-5-40} \\ 
\M_{4}(\Gamma_{0}(8))\subset\M_{4}(\Gamma_{0}(40)). \label{eqn-8-40}  \\
\M_{4}(\Gamma_{0}(7))\subset\M_{4}(\Gamma_{0}(14))\subset\M_{4}(\Gamma_{0}(28))\subset 
\M_{4}(\Gamma_{0}(56))  \label{eqn-7-56} \\ 
\M_{4}(\Gamma_{0}(8))\subset\M_{4}(\Gamma_{0}(56)). \label{eqn-8-56}
\end{align}
This implies the same inclusion relation for the bases, the space of Eisenstein 
forms of weight $4$ and the spaces of cusp forms of weight $4$. 

\subsection{Bases of $\E_{4}(\Gamma_{0}(\alpha\beta))$ and
  $\S_{4}(\Gamma_{0}(\alpha\beta))$ for $\alpha\beta=33,40,56$}  
\label{convolution_33_40_56-gen}

We apply the dimension formulae in \tM 's book \cite[Thrm 2.5.2,~p.~60]{miyake1989} or 
\cite[Prop.~6.1, p.~91]{wstein}
to deduce that $\text{dim}(\S_{4}(\Gamma_{0}(33))=10,\  
\text{dim}(\S_{4}(\Gamma_{0}(40))=14$ and 
$\text{dim}(\S_{4}(\Gamma_{0}(56))=20$. 
We use \autoref{dimension-Eisenstein} to infer that  
$\text{dim}(\E_{4}(\Gamma_{0}(33)))=4$ and  $\text{dim}(\E_{4}(\Gamma_{0}(40)))=
\text{dim}(\E_{4}(\Gamma_{0}(56)))=8$. 

We apply \autoref{ligozat_theorem} as mentioned in the third paragraph of 
\hyperref[convolution_alpha_beta-bases]{Subsection \ref*{convolution_alpha_beta-bases}} 
to explicitly determine as many elements of 
$\EuScript{B}_{S,33}$, $\EuScript{B}_{S,40}$ and $\EuScript{B}_{S,56}$  
as possible. Then we apply \hyperref[basis-remark]{Remark \ref*{basis-remark}} 
\textbf{(r2)} when selecting basis elements of a given space of cusp forms as 
stated in the proof of \autoref{basisCusp_a_b} (b).  
\begin{corollary} \label{basisCusp_33_40_56}
\begin{enumerate}
\item[\textbf{(a)}] The sets $\EuScript{B}_{E,33}=\{\,M(q^{t})\,\mid ~ t|33\,\}, 
\>\EuScript{B}_{E,40}=\{\, M(q^{t})\,\mid ~ t|40\,\}$ and 
$\EuScript{B}_{E,56}=\{\, M(q^{t})\,\mid ~ t|56\,\}$
are bases of $\E_{4}(\Gamma_{0}(33))$, $\E_{4}(\Gamma_{0}(40))$ and 
$\E_{4}(\Gamma_{0}(56))$, respectively. 
\item[\textbf{(b)}] Let $1\leq i\leq 10$, $1\leq j\leq 14$, $1\leq k\leq 20$ 
be positive integers. 

Let $\delta_{1}\in D(33)$ and 
$(r(i,\delta_{1}))_{i,\delta_{1}}$ be the 
\autoref{convolutionSums-3_11-table} of the powers of $\eta(\delta_{1} z)$. 

Let $\delta_{2}\in D(40)$ and 
$(r(j,\delta_{2}))_{j,\delta_{2}}$ be the 
\autoref{convolutionSums-5_8-table} of the powers of $\eta(\delta_{2} z)$. 

Let $\delta_{3}\in D(56)$ and 
$(r(k,\delta_{3}))_{k,\delta_{3}}$ be the 
\autoref{convolutionSums-7_8-table} of the powers of $\eta(\delta_{3} z)$. 

Let furthermore  
\begin{gather*}
\EuFrak{B}_{33,i}(q)=\underset{\delta_{1}|33}{\prod}\eta^{r(i,\delta_{1})}(\delta_{1}
z), \quad  
\EuFrak{B}_{40,j}(q)=\underset{\delta_{2}|40}{\prod}\eta^{r(j,\delta_{2})}(\delta_{2}
z), \\  
\EuFrak{B}_{56,k}(q)=\underset{\delta_{3}|56}{\prod}\eta^{r(k,\delta_{3})}(\delta_{3}z) 
\end{gather*} 
be selected elements of 
$\S_{4}(\Gamma_{0}(33)), \> 
\S_{4}(\Gamma_{0}(40))\> \text{ and }\> \S_{4}(\Gamma_{0}(56))$, 
respectively. 

Then the sets 
\begin{gather*}
\EuScript{B}_{S,33}=\{\,\EuFrak{B}_{33,i}(q)\,\mid ~ 1\leq i\leq 10\,\},\quad   
\EuScript{B}_{S,40}=\{\,\EuFrak{B}_{40,j}(q)\,\mid ~ 1\leq j\leq 14\,\},\\ 
\EuScript{B}_{S,56}=\{\,\EuFrak{B}_{56,k}(q)\,\mid ~ 1\leq k\leq 20\,\} 
\end{gather*}
are bases 
of $\S_{4}(\Gamma_{0}(33))$, $\S_{4}(\Gamma_{0}(40))$ and 
$\S_{4}(\Gamma_{0}(56))$, repectively.
\item[\textbf{(c)}] The sets
  $\EuScript{B}_{M,33}=\EuScript{B}_{E,33}\cup\EuScript{B}_{S,33}$, 
$\EuScript{B}_{M,40}=\EuScript{B}_{E,40}\cup\EuScript{B}_{S,40}$ and 
$\EuScript{B}_{M,56}=\EuScript{B}_{E,56}\cup\EuScript{B}_{S,56}$
constitute bases of $\M_{4}(\Gamma_{0}(33))$, $\M_{4}(\Gamma_{0}(40))$ and 
$\M_{4}(\Gamma_{0}(56))$, respectively.
\end{enumerate}
\end{corollary}
By \hyperref[basis-remark]{Remark \ref*{basis-remark}} \textbf{(r1)}, 
$\EuFrak{B}_{33,i}(q)$, $\EuFrak{B}_{40,j}(q)$ and $\EuFrak{B}_{56,k}(q)$    
can be expressed in the form 
$\underset{n=1}{\overset{\infty}{\sum}}\EuFrak{b}_{33,i}(n)q^{n}$, 
$\underset{n=1}{\overset{\infty}{\sum}}\EuFrak{b}_{40,j}(n)q^{n}$ and 
$\underset{n=1}{\overset{\infty}{\sum}}\EuFrak{b}_{56,k}(n)q^{n}$, respectively.

We observe that 
\begin{itemize}
	\item by \autoref{eqn-11-33} the basis element $\EuFrak{B}_{33,2}(q)$ 
	is in $\S_{4}(\Gamma_{0}(11))$ and is the only one. In addition, 
	$\EuFrak{B}_{33,6}(q)=\EuFrak{B}_{33,2}(q^{2})$. Hence, 
	$\EuFrak{b}_{33,6}(n)=\EuFrak{b}_{33,2}(\frac{n}{2})$. 
	
	\item the basis elements of $\S_{4}(\Gamma_{0}(40))$ have been determined 
	almost with respect to the inclusion relation \autoref{eqn-5-40}, except that 
	$\EuFrak{B}_{40,5}(q)$ results from the basis element of 
	$\S_{4}(\Gamma_{0}(8))$ according to \autoref{eqn-8-40}. 
	\item there is no element of $\S_{4}(\Gamma_{0}(7))$ which occurs as a 
	basis element of $\S_{4}(\Gamma_{0}(56))$. This indicates that an element 
	of $\S_{4}(\Gamma_{0}(7))$ cannot be determined when using 
	\autoref{ligozat_theorem}. Other than that, the inclusion relation 
	\autoref{eqn-7-56} and \autoref{eqn-8-56} preserve the bases.  
\end{itemize} 
\begin{proof} 
It follows immediately from \autoref{basisCusp_a_b}. 

In case \textbf{(a)}: the result is obtained when we set $n=1,3,11,33$, 
$n=1,2,4,5,8,10,20,40$ and $n=1,2,4,7,8,14,28,56$, respectively. 

In case \textbf{(b)}: the linear independence of the sets $\EuScript{B}_{S,33}$  
and $\EuScript{B}_{S,56}$ is proved by applying case 2 in the proof of 
\autoref{basisCusp_a_b} and by taking $n=1,2,3,4,5,6,7,8,9,10$ and  
$n=1,2,3,\ldots, 13,14$, respectively. Finally $\EuScript{B}_{S,40}$ is 
linearly independent by case 1 in the proof of \autoref{basisCusp_a_b} and by 
taking $n=1,2,3,\ldots, 19,20$. 

Therefore, we obtain the stated result. 
\end{proof}
 
\subsection{Evaluation of $W_{(\alpha,\beta)}(n)$ when $\alpha\beta=33,40,56$} 
\label{convolSum-w_33_40_56}

\begin{corollary} \label{lema-w_33_40_56}
We have  
\begin{multline}
( L(q) - 33\, L(q^{33}))^{2} 
 =  1024 + \sum_{n=1}^{\infty}\biggl(\, 
  \frac{2300736}{1271}\,\sigma_{3}(n) 
 - \frac{59459328}{77531}\,\sigma_{3}(\frac{n}{3})   \\
 + \frac{271016064}{1271}\,\sigma_{3}(\frac{n}{11})
- \frac{75206279808}{77531}\, \sigma_{3}(\frac{n}{33})
  - \frac{348480}{1271}\, \EuFrak{b}_{33,1}(n) 
 - \frac{14117760}{1271}\, \EuFrak{b}_{33,2}(n)   \\
- \frac{6573339072}{77531}\, \EuFrak{b}_{33,3}(n)    
- \frac{26803856448}{77531}\, \EuFrak{b}_{33,4}(n) 
- \frac{62014527936}{77531}\, \EuFrak{b}_{33,5}(n)   \\
- \frac{97134678144}{77531}\, \EuFrak{b}_{33,6}(n)    
- \frac{87378566400}{77531}\, \EuFrak{b}_{33,7}(n)
 - \frac{742808448}{1271}\, \EuFrak{b}_{33,8}(n)   \\
   + \frac{44352}{1271}\, \EuFrak{b}_{33,9}(n) 
  - \frac{4447872}{1271}\, \EuFrak{b}_{33,10}(n)  
\, \biggr)q^{n}, \label{convolSum-eqn-1_33}
\end{multline}
\begin{multline}
( 3\,L(q^{3}) - 11\, L(q^{11}))^{2} 
 = 64 + \sum_{n=1}^{\infty}\biggl(\, 
   - \frac{348480}{1271}\,\sigma_{3}(n)
  + \frac{106313472}{77531}\,\sigma_{3}(\frac{n}{3})    \\
  - \frac{34793088}{1271}\, \sigma_{3}(\frac{n}{11}) 
 + \frac{80875631232}{77531}\,\sigma_{3}(\frac{n}{33}) 
   + \frac{348480}{1271}\, \EuFrak{b}_{33,1}(n)    \\
   + \frac{3136320}{1271}\, \EuFrak{b}_{33,2}(n)  
 + \frac{1346173632}{77531}\, \EuFrak{b}_{33,3}(n)  
 + \frac{5361496704}{77531}\, \EuFrak{b}_{33,4}(n)    \\ 
 + \frac{11895235776}{77531}\, \EuFrak{b}_{33,5}(n)
 + \frac{17925551424}{77531}\, \EuFrak{b}_{33,6}(n) 
 + \frac{15428171520}{77531}\, \EuFrak{b}_{33,7}(n)    \\ 
  + \frac{127847808}{1271}\, \EuFrak{b}_{33,8}(n) 
   - \frac{44352}{1271}\, \EuFrak{b}_{33,9}(n) 
   + \frac{4447872}{1271}\, \EuFrak{b}_{33,10}(n)  
\,\biggr)q^{n},
\label{convolSum-eqn-3_11}
\end{multline}
\begin{multline}
( L(q) - 40\, L(q^{40}))^{2} 
 = 1521 + \sum_{n=1}^{\infty}\biggl(\, 
 \frac{26800}{117}\,\sigma_{3}(n) 
+ \frac{43520}{117}\,\sigma_{3}(\frac{n}{2})  \\
  + \frac{245120}{39}\,\sigma_{3}(\frac{n}{4}) 
  - \frac{26800}{117}\,\sigma_{3}(\frac{n}{5})
 - \frac{1766400}{13}\,\sigma_{3}(\frac{n}{8})
 - \frac{127760}{117}\,\sigma_{3}(\frac{n}{10})  \\
 - \frac{357440}{39}\,\sigma_{3}(\frac{n}{20})  
 + \frac{6558720}{13}\,\sigma_{3}(\frac{n}{40})
  + \frac{192224}{117}\,\EuFrak{b}_{40,1}(n)
  + \frac{439744}{117}\,\EuFrak{b}_{40,2}(n)  \\
  + \frac{304832}{39}\,\EuFrak{b}_{40,3}(n)  
 + \frac{1061120}{39}\,\EuFrak{b}_{40,4}(n)
 + \frac{41840}{3}\,\EuFrak{b}_{40,5}(n)
  - 15360\,\EuFrak{b}_{40,6}(n)  \\
 - \frac{24320}{3}\,\EuFrak{b}_{40,7}(n)  
 + \frac{1688320}{39}\,\EuFrak{b}_{40,8}(n)
 + 116800\,\EuFrak{b}_{40,9}(n)
- \frac{128000}{3}\,\EuFrak{b}_{40,10}(n)  \\
- \frac{485120}{3}\,\EuFrak{b}_{40,11}(n)  
- \frac{1130240}{3}\,\EuFrak{b}_{40,12}(n) 
- \frac{121280}{3}\,\EuFrak{b}_{40,13}(n)        
 - 69120\,\EuFrak{b}_{40,14}(n) 
\, \biggr)q^{n}, \label{convolSum-eqn-1_40}
\end{multline}
\begin{multline}
(5\, L(q^{5}) - 8\, L(q^{8}))^{2} 
 = 9 + \sum_{n=1}^{\infty}\biggl(\, 
   \frac{5920}{117}\,\sigma_{3}(n)
 - \frac{76000}{117}\,\sigma_{3}(\frac{n}{2}) \\
 - \frac{16960}{39}\,\sigma_{3}(\frac{n}{4}) 
 + \frac{668000}{117}\,\sigma_{3}(\frac{n}{5}) 
 + \frac{721920}{13}\,\sigma_{3}(\frac{n}{8}) 
 - \frac{8240}{117}\,\sigma_{3}(\frac{n}{10}) \\
 - \frac{95360}{39}\,\sigma_{3}(\frac{n}{20}) 
- \frac{721920}{13}\,\sigma_{3}(\frac{n}{40}) 
 - \frac{5920}{117}\,\EuFrak{b}_{40,1}(n)
 + \frac{22720}{117}\,\EuFrak{b}_{40,2}(n)  \\
 - \frac{59200}{39} \,\EuFrak{b}_{40,3}(n)
 + \frac{12800}{39}\,\EuFrak{b}_{40,4}(n)
- \frac{38800}{3} \,\EuFrak{b}_{40,5}(n)
 + 7680\,\EuFrak{b}_{40,6}(n)   \\
- \frac{47360}{3}\,\EuFrak{b}_{40,7}(n)
- \frac{505088}{39}\,\EuFrak{b}_{40,8}(n)
 - 67520\,\EuFrak{b}_{40,9}(n)
- \frac{12800}{3}\,\EuFrak{b}_{40,10}(n)  \\
+ \frac{113920}{3}\,\EuFrak{b}_{40,11}(n)
+ \frac{298240}{3}\,\EuFrak{b}_{40,12}(n)
+ \frac{63040}{3}\,\EuFrak{b}_{40,13}(n)
+ 69120 \,\EuFrak{b}_{40,14}(n) \,\biggr)q^{n},
\label{convolSum-eqn-5_8}
\end{multline}
\begin{multline}
( L(q) - 56\, L(q^{56}))^{2} 
 = 3025 + \sum_{n=1}^{\infty}\biggl(\, \frac{1284}{5}\,\sigma_{3}(n) 
    - 420\, \sigma_{3}(\frac{n}{2}) 
    + \frac{31584}{5}\,\sigma_{3}(\frac{n}{4}) \\
   - \frac{1764}{5}\,\sigma_{3}(\frac{n}{7})  
   - \frac{32256}{5}\, \sigma_{3}(\frac{n}{8})
    - 588\, \sigma_{3}(\frac{n}{14}) 
    - \frac{51744}{5}\,\sigma_{3}(\frac{n}{28})
    + \frac{3687936}{5}\, \sigma_{3}(\frac{n}{56})  \\
    + \frac{11916}{5}\,\EuFrak{b}_{56,1}(n)  
 + \frac{92604}{5} \,\EuFrak{b}_{56,2}(n)  
  + 29568 \,\EuFrak{b}_{56,3}(n)  
  + \frac{1140216}{5} \,\EuFrak{b}_{56,4}(n)  \\
 - 411936  \,\EuFrak{b}_{56,5}(n)  
 + \frac{2557632}{5} \,\EuFrak{b}_{56,6}(n)  
 + 223608 \,\EuFrak{b}_{56,7}(n)    
 + 3998400 \,\EuFrak{b}_{56,8}(n)  \\
 + 4042752  \,\EuFrak{b}_{56,9}(n)  
+ \frac{145152}{5} \,\EuFrak{b}_{56,10}(n) 
  - 8064  \,\EuFrak{b}_{56,11}(n)  
 - 48384 \,\EuFrak{b}_{56,12}(n)    \\
    + \frac{532224}{5} \,\EuFrak{b}_{56,14}(n)  
 +  161280 \,\EuFrak{b}_{56,15}(n)    
- \frac{225792}{5} \,\EuFrak{b}_{56,16}(n)  
 + 129024  \,\EuFrak{b}_{56,17}(n)     \\
+ \frac{2515968}{5} \,\EuFrak{b}_{56,18}(n)  
 + 1354752  \,\EuFrak{b}_{56,19}(n)  
 - 225792 \,\EuFrak{b}_{56,20}(n)   
\,\biggr)q^{n}, \label{convolSum-eqn-1_56}
\end{multline}
\begin{multline}
( 7\,L(q^{7}) - 8\, L(q^{8}))^{2} 
 = 1 + \sum_{n=1}^{\infty}\biggl(\, 
  - \frac{308}{25}\,\sigma_{3}(n)   
  + \frac{1876}{25}\,\sigma_{3}(\frac{n}{2})   
 - \frac{40096}{25}\, \sigma_{3}(\frac{n}{4})  \\
 + \frac{285908}{25}\, \sigma_{3}(\frac{n}{7})  
 \frac{409088}{25}\,\sigma_{3}(\frac{n}{8})   
 - \frac{27076}{25}\,\sigma_{3}(\frac{n}{14})   
 - \frac{60704}{25}\,\sigma_{3}(\frac{n}{28})  \\
- \frac{562688}{25}\,\sigma_{3}(\frac{n}{56})    
 + \frac{308}{25}\,\EuFrak{b}_{56,1}(n)    
 + \frac{2436}{25}\,\EuFrak{b}_{56,2}(n)   
 +\frac{11648}{25}\,\EuFrak{b}_{56,3}(n)    \\
 +\frac{121352}{25}\,\EuFrak{b}_{56,4}(n)   
 - \frac{101472}{25}\,\EuFrak{b}_{56,5}(n)    
 +\frac{288064}{25}\,\EuFrak{b}_{56,6}(n)   
 - \frac{87864}{25}\,\EuFrak{b}_{56,7}(n)   \\
+\frac{2190912}{25}\,\EuFrak{b}_{56,8}(n)    
+\frac{1821312}{25}\,\EuFrak{b}_{56,9}(n)   
+ \frac{201984}{25}\,\EuFrak{b}_{56,10}(n)   
+ 29568 \,\EuFrak{b}_{56,11}(n)             \\
\frac{284928}{5} \,\EuFrak{b}_{56,12}(n)
 - 59136 \,\EuFrak{b}_{56,13}(n) 
- \frac{1512192}{25}\,\EuFrak{b}_{56,14}(n)  
 +59136 \,\EuFrak{b}_{56,15}(n)              \\
+\frac{6724096}{25} \,\EuFrak{b}_{56,16}(n)   
 +118272 \,\EuFrak{b}_{56,17}(n)
+\frac{1616896}{25}\, \,\EuFrak{b}_{56,18}(n) \\  
 +430080\, \,\EuFrak{b}_{56,19}(n) 
 +53760\, \,\EuFrak{b}_{56,20}(n)  
\,\biggr)q^{n}.
\label{convolSum-eqn-7_8}
\end{multline}
\end{corollary}
\begin{proof} 
These identities follow immediately 
on taking $(\alpha,\beta)=(1,33)$, $(3,11)$, $(1,40)$, $(5,8)$, $(1,56)$, 
$(7,8)$ in  
\hyperref[convolution-lemma_a_b]{Lemma \ref*{convolution-lemma_a_b}}.
In case $\alpha\beta=40$ we take all $n$ in $\{1,2,\ldots,20,40,80\}$ 
to obtain a system of $22$ linear equations with unknowns $X_{\delta}$ and $Y_{j}$, 
where $\delta\in D(40)$ and $1\leq j\leq 14$.
\end{proof}
We are now prepared to state and to prove our main result of this section. 
\begin{corollary} \label{convolSum-theor-w_33_40_56}
Let $n$ be a positive integer. Then 
\begin{align}
W_{(1,33)}(n)   = &
- \frac{13859}{335544}\sigma_{3}(n) 
   + \frac{51614}{2558523}\,\sigma_{3}(\frac{n}{3})    
  - \frac{7129}{1271}\,\sigma_{3}(\frac{n}{11})  
  + \frac{60271327}{1860744}\,\sigma_{3}(\frac{n}{33}) \notag \\  & 
  + (\frac{1}{24}-\frac{1}{132}n)\sigma(n) 
+ (\frac{1}{24}-\frac{1}{4}n)\sigma(\frac{n}{33}) 
 + \frac{55}{7626}\,\EuFrak{b}_{33,1}(n)   
  + \frac{4085}{13981}\,\EuFrak{b}_{33,2}(n)  \notag \\  &   
 + \frac{11412047}{5117046}\,\EuFrak{b}_{33,3}(n)   
 + \frac{15511491}{1705682}\,\EuFrak{b}_{33,4}(n)   
  + \frac{35888037}{1705682}\,\EuFrak{b}_{33,5}(n)     \notag \\  &   
  + \frac{28106099}{852841}\,\EuFrak{b}_{33,6}(n) 
  + \frac{25283150}{852841}\,\EuFrak{b}_{33,7}(n)    
   + \frac{214933}{13981}\,\EuFrak{b}_{33,8}(n)     \notag \\  & 
   - \frac{7}{7626}\,\EuFrak{b}_{33,9}(n)    
   + \frac{117}{1271}\,\EuFrak{b}_{33,10}(n),   
\label{convolSum-theor-w_1_33}  
\end{align} 
\begin{align}
W_{(3,11)}(n)  =  &
     \frac{55}{7626}\,\sigma_{3}(n) 
   + \frac{12869}{620248}\,\sigma_{3}(\frac{n}{3}) 
   + \frac{15089}{10168}\,\sigma_{3}(\frac{n}{11})
  - \frac{6382231}{232593}\,\sigma_{3}(\frac{n}{33})   \notag \\  & 
   + (\frac{1}{24} -\frac{1}{44}n)\sigma(\frac{1}{3}n) 
  + (\frac{1}{24}-\frac{1}{12}n)\sigma(\frac{n}{11})   
  - \frac{55}{7626}\, \EuFrak{b}_{33,1}(n)  
  - \frac{165}{2542}\, \EuFrak{b}_{33,2}(n)   \notag \\  & 
  - \frac{2337107}{5117046}\, \EuFrak{b}_{33,3}(n) 
  - \frac{1551359}{852841}\, \EuFrak{b}_{33,4}(n) 
  - \frac{6883817}{1705682}\, \EuFrak{b}_{33,5}(n)   \notag \\  & 
  - \frac{943053}{155062}\, \EuFrak{b}_{33,6}(n)  
  - \frac{4464170}{852841}\, \EuFrak{b}_{33,7}(n)  
  - \frac{3363}{1271}\, \EuFrak{b}_{33,8}(n)      \notag \\  & 
  + \frac{7}{7626}\, \EuFrak{b}_{33,9}(n)     
  - \frac{117}{1271}\, \EuFrak{b}_{33,10}(n),  
\label{convolSum-theor-w_3_11}  
 \end{align}
\begin{align}
W_{(1,40)}(n)  =  & 
  \frac{1}{4212}\,\sigma_{3}(n)
   - \frac{17}{2106}\,\sigma_{3}(\frac{n}{2})
   - \frac{383}{2808}\,\sigma_{3}(\frac{n}{4})
  + \frac{335}{67392}\,\sigma_{3}(\frac{n}{5})  
  + \frac{115}{39}\,\sigma_{3}(\frac{n}{8})    \notag \\  & 
  + \frac{1597}{67392}\,\sigma_{3}(\frac{n}{10})
   + \frac{1117}{5616}\,\sigma_{3}(\frac{n}{20})
   - \frac{34}{13}\,\sigma_{3}(\frac{n}{40}) 
 + (\frac{1}{24}-\frac{1}{160}n)\sigma(n)   \notag \\  & 
 + (\frac{1}{24}-\frac{1}{4}n)\sigma(\frac{n}{40})    
  - \frac{6007}{168480}\,\EuFrak{b}_{40,1}(n)
  - \frac{6871}{84240}\,\EuFrak{b}_{40,2}(n) 
  - \frac{4763}{28080}\,\EuFrak{b}_{40,3}(n)  \notag \\  & 
   - \frac{829}{1404}\,\EuFrak{b}_{40,4}(n)
   - \frac{523}{1728}\,\EuFrak{b}_{40,5}(n)
   + \frac{1}{3}\,\EuFrak{b}_{40,6}(n)    
   + \frac{19}{108}\,\EuFrak{b}_{40,7}(n)   \notag \\  &  
  - \frac{1319}{1404}\,\EuFrak{b}_{40,8}(n)
   - \frac{365}{144}\,\EuFrak{b}_{40,9}(n)
   + \frac{25}{27}\,\EuFrak{b}_{40,10}(n)  
   + \frac{379}{108}\,\EuFrak{b}_{40,11}(n)   \notag \\  & 
   + \frac{883}{108}\,\EuFrak{b}_{40,12}(n)  
   + \frac{379}{432}\,\EuFrak{b}_{40,13}(n)
   + \frac{3}{2}\,\EuFrak{b}_{40,14}(n),
\label{convolSum-theor-w_1_40}  
\end{align} 
\begin{align}
W_{(5,8)}(n)  =  & 
 - \frac{37}{33696}\,\sigma_{3}(n)
 +  \frac{475}{33696}\,\sigma_{3}(\frac{n}{2})
 +  \frac{53}{5616}\,\sigma_{3}(\frac{n}{4})
 +  \frac{425}{67392}\,\sigma_{3}(\frac{n}{5})  
 - \frac{34}{39}\,\sigma_{3}(\frac{n}{8})     \notag \\  & 
 +  \frac{103}{67392}\,\sigma_{3}(\frac{n}{10})
 + \frac{149}{2808}\,\sigma_{3}(\frac{n}{20})
 +  \frac{47}{39}\,\sigma_{3}(\frac{n}{40})
 + (\frac{1}{24}-\frac{1}{32}n)\sigma(\frac{n}{5}) \notag \\  & 
 + (\frac{1}{24}-\frac{1}{20}n)\sigma(\frac{n}{8})    
  + \frac{37}{33696}\,\EuFrak{b}_{40,1}(n)
  - \frac{71}{16848}\,\EuFrak{b}_{40,2}(n)  
 + \frac{185}{5616}\,\EuFrak{b}_{40,3}(n)   \notag \\  & 
  - \frac{5}{702}\,\EuFrak{b}_{40,4}(n)  
 + \frac{485}{1728}\,\EuFrak{b}_{40,5}(n)
 - \frac{1}{6}\,\EuFrak{b}_{40,6}(n)  
  + \frac{37}{108}\,\EuFrak{b}_{40,7}(n)
 + \frac{1973}{7020}\,\EuFrak{b}_{40,8}(n)  \notag \\  & 
  + \frac{211}{144}\,\EuFrak{b}_{40,9}(n)
 + \frac{5}{54}\,\EuFrak{b}_{40,10}(n)  
  - \frac{89}{108}\,\EuFrak{b}_{40,11}(n)
 - \frac{233}{108}\,\EuFrak{b}_{40,12}(n)  \notag \\  & 
 - \frac{197}{432}\,\EuFrak{b}_{40,13}(n)
 - \frac{3}{2}\,\EuFrak{b}_{40,14}(n), 
\label{convolSum-theor-w_5_8}  
\end{align}
\begin{align}
W_{(1,56)}(n)   =  & 
- \frac{1}{3840}\,\sigma_{3}(n)  
  + \frac{5}{768}\,\sigma_{3}(\frac{n}{2}) 
  - \frac{47}{480}\,\sigma_{3}(\frac{n}{4})   
  + \frac{7}{1280}\,\sigma_{3}(\frac{n}{7}) 
  + \frac{1}{10}\,\sigma_{3}(\frac{n}{8})  \notag \\  & 
  + \frac{7}{768}\,\sigma_{3}(\frac{n}{14})   
  + \frac{77}{480}\,\sigma_{3}(\frac{n}{28})          
  + \frac{7}{30}\,\sigma_{3}(\frac{n}{56})  
  + (\frac{1}{24} - \frac{1}{224}n)\sigma(n) \notag \\  & 
  + (\frac{1}{24} - \frac{1}{4}n)\sigma(\frac{n}{56}) 
  - \frac{331}{8960} \,\EuFrak{b}_{56,1}(n) 
  - \frac{7717}{26880}\,\EuFrak{b}_{56,2}(n) 
   - \frac{11}{24} \,\EuFrak{b}_{56,3}(n)   \notag \\  & 
  - \frac{6787}{1920} \,\EuFrak{b}_{56,4}(n)  
  + \frac{613}{96} \,\EuFrak{b}_{56,5}(n)     
  - \frac{1903}{240} \,\EuFrak{b}_{56,6}(n)  
  - \frac{1331}{384} \,\EuFrak{b}_{56,7}(n)  \notag \\  & 
  - \frac{2975}{48} \,\EuFrak{b}_{56,8}(n)    
  - \frac{188}{3} \,\EuFrak{b}_{56,9}(n)    
    - \frac{9}{20} \,\EuFrak{b}_{56,10}(n) 
   + \frac{1}{8} \,\EuFrak{b}_{56,11}(n)  
   + \frac{3}{4} \,\EuFrak{b}_{56,12}(n)    \notag \\  & 
  - \frac{33}{20} \,\EuFrak{b}_{56,14}(n)  
   - \frac{5}{2} \,\EuFrak{b}_{56,15}(n) 
   \frac{7}{10} \,\EuFrak{b}_{56,16}(n)    
    - 2 \,\EuFrak{b}_{56,17}(n)  
  - \frac{39}{5} \,\EuFrak{b}_{56,18}(n)   \notag \\  & 
   - 21 \,\EuFrak{b}_{56,19}(n)          
   + \frac{7}{2} \,\EuFrak{b}_{56,20}(n),  
\label{convolSum-theor-w_1_56}  
\end{align} 
\begin{align}
W_{(7,8)}(n)  =  & 
\frac{11}{57600}\,\sigma_{3}(n) 
  - \frac{67}{57600}\,\sigma_{3}(\frac{n}{2}) 
  + \frac{179}{7200}\,\sigma_{3}(\frac{n}{4})  
  + \frac{289}{57600}\,\sigma_{3}(\frac{n}{7})  
  - \frac{7}{450}\,\sigma_{3}(\frac{n}{8})     \notag \\  & 
  + \frac{967}{57600}\,\sigma_{3}(\frac{n}{14})
  + \frac{271}{7200}\,\sigma_{3}(\frac{n}{28}) 
  + \frac{157}{450}\,\sigma_{3}(\frac{n}{56})  
  + (\frac{1}{24} - \frac{1}{32}n)\sigma(\frac{n}{7}) \notag \\  & 
  + (\frac{1}{24} - \frac{1}{28}n)\sigma(\frac{n}{8})   
  - \frac{11}{57600} \,\EuFrak{b}_{56,1}(n)  
  - \frac{29}{19200} \,\EuFrak{b}_{56,2}(n) 
  - \frac{13}{1800} \,\EuFrak{b}_{56,3}(n)    \notag \\  & 
  - \frac{2167}{28800} \,\EuFrak{b}_{56,4}(n)  
  + \frac{151}{2400} \,\EuFrak{b}_{56,5}(n)   
   - \frac{643}{3600} \,\EuFrak{b}_{56,6}(n)  
  + \frac{523}{9600} \,\EuFrak{b}_{56,7}(n)  \notag \\  & 
 - \frac{11411}{8400} \,\EuFrak{b}_{56,8}(n) 
 - \frac{1581}{1400} \,\EuFrak{b}_{56,9}(n) 
  - \frac{263}{2100} \,\EuFrak{b}_{56,10}(n)  
  - \frac{11}{24} \,\EuFrak{b}_{56,11}(n)   \notag \\  & 
  - \frac{53}{60} \,\EuFrak{b}_{56,12}(n)   
  + \frac{11}{12} \,\EuFrak{b}_{56,13}(n) 
  + \frac{1969}{2100} \,\EuFrak{b}_{56,14}(n)  
  - \frac{11}{12} \,\EuFrak{b}_{56,15}(n)     \notag \\  & 
 - \frac{13133}{3150} \,\EuFrak{b}_{56,16}(n)  
 - \frac{11}{6} \,\EuFrak{b}_{56,17}(n) 
 - \frac{1579}{1575} \,\EuFrak{b}_{56,18}(n)  
 - \frac{20}{3} \,\EuFrak{b}_{56,19}(n)       \notag \\  & 
  - \frac{5}{6} \,\EuFrak{b}_{56,20}(n). 
\label{convolSum-theor-w_7_8}  
 \end{align}
\end{corollary}
\begin{proof} 
These identities follow from \autoref{convolution_a_b} when we set  
$(\alpha,\beta)=(1,33),(3,11)$, $(1,40),(5,8)$, $(1,56),(7,8)$.
\end{proof}


\section{Re-evaluation of the convolution sums for $\alpha\beta=10,11,12,15,24$}
\label{re-evaluation}
We revisit the convolution sums established by 
\begin{itemize}
	\item \eR \cite[Thrm 1.1]{royer}, and \sC\ and \dY\ \cite[Thrm 2.1]{cooper_ye2014} for 
	$\alpha\beta=10$ 
	\item \eR \cite[Thrm 1.3]{royer} for $\alpha\beta=11$ 
	\item \aA\ et al.\ \cite{alaca_alaca_williams2006,alaca_alaca_williams2007a} for $\alpha\beta=12$, $24$ 
	\item \bR\ and \bS\ \cite{ramakrishnan_sahu} for $\alpha\beta=15$
\end{itemize}
using modular forms. The obtained results in each case are immediate corollaries 
of \autoref{convolution_a_b} 
and improve the previous ones since we use the exact number of basis elements 
of the space of cusp forms in case of $\alpha\beta=12$, $24$. The improvement of the previous results in case of $\alpha\beta=11$, $15$ is obvious.

Since $\alpha\beta=10=2\cdot 5$ and because of  
\autoref{eqn-5-40} it holds that $\EuFrak{B}_{40,2}(q)=\EuFrak{B}_{40,1}(q^{2})$, 
and therefore $\EuFrak{b}_{40,2}(n)=\EuFrak{b}_{40,1}(\frac{n}{2})$. Our 
third basis element of the space $\S_{4}(\Gamma_{0}(10))$ is different 
from the one used by \sC\ and \dY\ \cite{cooper_ye2014}, 
which explains the difference in the two results. However, since the change 
of basis is an automorphism, both results are the same.

In addition to the basis element $\EuFrak{B}_{33,2}(q)$ of the space 
$\S_{4}(\Gamma_{0}(11))$, we use 
$\EuFrak{B}'_{33,1}(q)=\eta^{2}(z)\eta^{2}(11z)=\underset{n=1}{\overset{\infty}{\sum}}\EuFrak{b}'_{33,1}(n)q^{n}$ which is a basis element of 
$\S_{2}(\Gamma_{0}(11))$. 

\bR\ and \bS\ achieve the evaluation of the convolution sums for $\alpha\beta=15$ 
using a basis which contains one cusp form of weight $2$. 
We consider the following $\eta$-quotients as basis elements of the space 
$\S_{4}(\Gamma_0(15))$ 
\begin{longtable}{rclcrcl}
$\EuFrak{B}_{15,1}(q)$ & = & $\eta^{4}(z)\eta^{4}(5z)$ & ~ &
$\EuFrak{B}_{15,2}(q)$ & = & $\eta^{2}(z)\eta^{2}(3z)\eta^{2}(5z)\eta^{2}(15z)$ \\
 ~ \\
$\EuFrak{B}_{15,3}(q)$ & = & $\eta^{4}(3z)\eta^{4}(15z)$  & ~ &
$\EuFrak{B}_{15,4}(q)$ & = & $\frac{\eta^{3}(z)\eta(3z)\eta^{7}(15z)}{\eta^{3}(5z)}$. 
\end{longtable}
The following $\eta$-quotients build a basis of $\S_{4}(\Gamma_0(24))$. 
\begin{longtable}{rclcrcl}
$\EuFrak{B}_{24,1}(q)$ & = & $\eta^{2}(z)\eta^{2}(2z)\eta^{2}(3z)\eta^{2}(6z)$ \\ 
$\EuFrak{B}_{24,2}(q)$ & = & $\eta^{2}(2z)\eta^{2}(4z)\eta^{2}(6z)\eta^{2}(12z)$ & ~&
$\EuFrak{B}_{24,3}(q)$ & = & $\frac{\eta^{4}(2z)\eta^{6}(12z)}{\eta^{2}(4z)}$  \\
$\EuFrak{B}_{24,4}(q)$ & = & $\eta^{2}(4z)\eta^{2}(8z)\eta^{2}(12z)\eta^{2}(24z)$ \\ 
  ~  \\
$\EuFrak{B}_{24,5}(q)$ & = & $\frac{\eta^{2}(2z)\eta^{2}(8z)\eta^{6}(12z)\eta^{2}(24z)}{\eta^{2}(4z)\eta^{2}(6z)}$ & ~ &
$\EuFrak{B}_{24,6}(q)$ & = & $\frac{\eta^{4}(4z)\eta^{6}(24z)}{\eta^{2}(8z)}$ \\
 ~ \\
$\EuFrak{B}_{24,7}(q)$ & = & $\frac{\eta^{2}(2z)\eta^{4}(12z)\eta^{6}(24z)}{\eta^{2}(6z)\eta^{2}(8z)}$  \\ 
  ~  \\
$\EuFrak{B}_{24,8}(q)$ & = & $\frac{\eta^{3}(3z)\eta^{3}(4z)\eta^{9}(24z)}{\eta(z)\eta(2z)\eta(6z)\eta^{3}(8z)\eta(12z)}$ 
\end{longtable}
Note that the space $\S_{4}(\Gamma_0(12))$ is a subspace of the space 
$\S_{4}(\Gamma_0(24))$. 

Basis elements of $\S_{4}(\Gamma_0(15))$ and $\S_{4}(\Gamma_0(24))$ can be 
expressed in the form 
$\EuFrak{B}_{15,j}(q)=\underset{n=1}{\overset{\infty}{\sum}}\EuFrak{b}_{15,j}(n)q^{n}$ 
and 
$\EuFrak{B}_{24,l}(q)=\underset{n=1}{\overset{\infty}{\sum}}\EuFrak{b}_{24,l}(n)q^{n}$, 
where $1\leq j\leq 4$ and $1\leq l\leq 8$, respectively. 
It holds that $\EuFrak{B}_{15,3}(q)=\EuFrak{B}_{15,1}(q^{2})$,  and  
$\EuFrak{B}_{24,2}(q)=\EuFrak{B}_{24,1}(q^{2})$, 
$\EuFrak{B}_{24,4}(q)=\EuFrak{B}_{24,1}(q^{4})$  and   
$\EuFrak{B}_{24,6}(q)=\EuFrak{B}_{24,3}(q^{2})$.  
Therefore $\EuFrak{b}_{15,3}(n)=\EuFrak{b}_{15,1}(\frac{n}{2})$,  and 
$\EuFrak{b}_{24,2}(n)=\EuFrak{b}_{24,1}(\frac{n}{2})$,
$\EuFrak{b}_{24,4}(n)=\EuFrak{b}_{24,1}(\frac{n}{4})$  and  
$\EuFrak{b}_{24,6}(n)=\EuFrak{b}_{24,3}(\frac{n}{2})$.
\begin{corollary} \label{lema-w_1-4_10}
We have  
\begin{multline}
( L(q) - 10\, L(q^{10}))^{2}  
 = 81 + \sum_{n=1}^{\infty}\biggl(\, 
  \frac{2640}{13}\,\sigma_{3}(n)
- \frac{1920}{13}\,\sigma_{3}(\frac{n}{2})
- \frac{12000}{13}\,\sigma_{3}(\frac{n}{5})     \\
+ \frac{264000}{13}\,\sigma_{3}(\frac{n}{10})   
 + \frac{2976}{13}\,\EuFrak{b}_{40,1}(n) 
+ \frac{14400}{13}\,\EuFrak{b}_{40,2}(n)       
 - 960\,\EuFrak{b}_{40,3}(n)  
 \biggr)q^{n}, \label{convolSum-eqn-1_10}  
\end{multline}
\begin{multline}
( 2\,L(q^{2}) - 5\, L(q^{5}))^{2}  
 = 9 + \sum_{n=1}^{\infty}\biggl(\, 
- \frac{480}{13}\,\sigma_{3}(n)
+ \frac{10560}{13}\,\sigma_{3}(\frac{n}{2})
+ \frac{66000}{13}\,\sigma_{3}(\frac{n}{5})    \\
- \frac{48000}{13}\,\sigma_{3}(\frac{n}{10})   
+  \frac{480}{13}\,\EuFrak{b}_{40,1}(n) 
 - \frac{576}{13}\,\EuFrak{b}_{40,2}(n)
+ 960\,\EuFrak{b}_{40,3}(n)  
 \biggr)q^{n}, \label{convolSum-eqn-2_5}  
\end{multline}
\begin{multline}
( L(q) - 11\, L(q^{11}))^{2}  
 = 100 + \sum_{n=1}^{\infty}\biggl(\, 
 \frac{6240}{49}\,\sigma_{3}(n)
 + \frac{5524320}{49}\,\sigma_{3}(\frac{n}{11}) 
+ \frac{17280}{49}\, \EuFrak{b}'_{33,1}(n)  \\
 + \frac{77184}{49}\, \EuFrak{b}_{33,2}(n) 
 \biggr)q^{n}, \label{convolSum-eqn-1_11}  
\end{multline}
\begin{multline}
( L(q) - 12\, L(q^{12}))^{2}  
 = 121 + \sum_{n=1}^{\infty}\biggl(\, 
  \frac{1056}{5}\,\sigma_{3}(n)
 - \frac{432}{5}\,\sigma_{3}(\frac{n}{2})
- \frac{1296}{5}\,\sigma_{3}(\frac{n}{3})   \\
- \frac{2304}{5}\,\sigma_{3}(\frac{n}{4})
- \frac{3888}{5}\,\sigma_{3}(\frac{n}{6})
 + \frac{152064}{5}\,\sigma_{3}(\frac{n}{12})
 + \frac{1584}{5}\,\EuFrak{b}_{24,1}(n)  
 + \frac{4896}{5}\,\EuFrak{b}_{24,2}(n)    \\
 + 864\,\EuFrak{b}_{24,3}(n)  
 \biggr)q^{n}, \label{convolSum-eqn-1_12}  
\end{multline}
\begin{multline}
( 3\,L(q^{3}) - 4\,L(q^{4}))^{2}  
 = 1 + \sum_{n=1}^{\infty}\biggl(\, 
 - \frac{144}{5}\,\sigma_{3}(n)
 - \frac{432}{5}\,\sigma_{3}(\frac{n}{2}) 
 + \frac{9504}{5}\,\sigma_{3}(\frac{n}{3})    \\
 + \frac{16896}{5}\,\sigma_{3}(\frac{n}{4}) 
- \frac{3888}{5}\,\sigma_{3}(\frac{n}{6}) 
- \frac{20736}{5}\,\sigma_{3}(\frac{n}{12}) 
 + \frac{144}{5}\,\EuFrak{b}_{24,1}(n)  
 + \frac{2016}{5}\,\EuFrak{b}_{24,2}(n)    \\
 - 864\,\EuFrak{b}_{24,3}(n)  
 \biggr)q^{n}, \label{convolSum-eqn-3_4}  
\end{multline}
\begin{multline}
( L(q) - 15\, L(q^{15}))^{2}  
 = 196 + \sum_{n=1}^{\infty}\biggl(\, 
  \frac{2976}{13}\,\sigma_{3}(n)
 - \frac{3456}{13}\,\sigma_{3}(\frac{n}{3})
- \frac{144000}{13}\,\sigma_{3}(\frac{n}{5})      \\
+ \frac{756000}{13}\,\sigma_{3}(\frac{n}{15})
+  \frac{5760}{13}\,\EuFrak{b}_{15,1}(n) 
+ 2304\,\EuFrak{b}_{15,2}(n) 
 + \frac{48384}{13}\,\EuFrak{b}_{15,3}(n)    
 - 3456\,\EuFrak{b}_{15,4}(n) 
 \biggr)q^{n}, \label{convolSum-eqn-1_15}  
\end{multline}
\begin{multline}
( 3\,L(q^{3}) - 5\, L(q^{5}))^{2}  
 = 4 + \sum_{n=1}^{\infty}\biggl(\, 
 - \frac{576}{13}\,\sigma_{3}(n)
+  \frac{25056}{13}\,\sigma_{3}(\frac{n}{3})
+ \frac{204000}{13}\,\sigma_{3}(\frac{n}{5})    \\
- \frac{216000}{13}\,\sigma_{3}(\frac{n}{15})
 +  \frac{576}{13}\,\EuFrak{b}_{15,1}(n)
+  576\,\EuFrak{b}_{15,2}(n)
 + \frac{8640}{13}\,\EuFrak{b}_{15,3}(n) 
+ 3456\,\EuFrak{b}_{15,4}(n)  
 \biggr)q^{n}. \label{convolSum-eqn-3_5}  
\end{multline}
\begin{multline}
( L(q) - 24\, L(q^{24}))^{2}  
 = 529 + \sum_{n=1}^{\infty}\biggl(\, 
   672\,\sigma_{3}(n)  
 + \frac{33264}{5}\,\sigma_{3}(\frac{n}{2})
  - 576\,\sigma_{3}(\frac{n}{3})            \\
- \frac{36576}{5}\,\sigma_{3}(\frac{n}{4})  
- \frac{35424}{5}\,\sigma_{3}(\frac{n}{6})
 - \frac{4608}{5}\,\sigma_{3}(\frac{n}{8})
 + \frac{27936}{5}\,\sigma_{3}(\frac{n}{12})
+ \frac{649728}{5}\,\sigma_{3}(\frac{n}{24})   \\
  + 432\,\EuFrak{b}_{24,1}(n)    
- \frac{44064}{5}\,\EuFrak{b}_{24,2}(n)  
  - 8640 \,\EuFrak{b}_{24,3}(n)  
- \frac{508608}{5}\,\EuFrak{b}_{24,4}(n)  \\
 - 55296\,\EuFrak{b}_{24,5}(n)    
 - 316224\,\EuFrak{b}_{24,6}(n)   
 - 276480\,\EuFrak{b}_{24,7}(n)    
 - 857088\,\EuFrak{b}_{24,8}(n)   
 \biggr)q^{n}, \label{convolSum-eqn-1_24}  
\end{multline}
\begin{multline}
( 3\,L(q^{3}) - 8\, L(q^{8}))^{2}  
 = 25 + \sum_{n=1}^{\infty}\biggl(\, 
 + \frac{864}{5}\,\sigma_{3}(\frac{n}{2})      
 + 2016\,\sigma_{3}(\frac{n}{3})
- \frac{2016}{5}\,\sigma_{3}(\frac{n}{4})      \\
- \frac{3024}{5}\,\sigma_{3}(\frac{n}{6})  
+ \frac{72192}{5}\,\sigma_{3}(\frac{n}{8}) 
- \frac{6624}{5}\,\sigma_{3}(\frac{n}{12}) 
- \frac{41472}{5}\,\sigma_{3}(\frac{n}{24})       
 - \frac{864}{5}\,\EuFrak{b}_{24,2}(n)      \\
 - 1296\,\EuFrak{b}_{24,3}(n)           
 - \frac{7488}{5}\,\EuFrak{b}_{24,4}(n)         
 - 15552\,\EuFrak{b}_{24,6}(n)          
 - 27648\,\EuFrak{b}_{24,8}(n)   
 \biggr)q^{n}. \label{convolSum-eqn-3_8}  
\end{multline}
\end{corollary}
In the case of the evaluation of $W_{(1,1)}(n)$, we observe, using 
\hyperref[evalConvolClass-lema-1]{Lemma \ref*{evalConvolClass-lema-1}}, that 
for all $\alpha,\beta\in\mathbb{N}$ it hods that 
\begin{equation}  \label{evalConvolClass-lema-1-eqn-0}
0  =  ( \alpha\,L(q^{\alpha}) - \alpha\, L(q^{\alpha}))^{2}\in
\M_{4}(\Gamma_{0}(\alpha\beta)). 
\end{equation}
\begin{corollary} \label{cor-1-4_10}
Let $n$ be a positive integer. Then 
\begin{align}  \label{besge-glaisher-ramanujan-1}
 \forall \alpha\in\mathbb{N} \quad W_{(\alpha,\alpha)}(n)  = & 
 \frac{5}{12}\sigma_{3}(\frac{n}{\alpha}) + (\frac{1}{12}-\frac{1}{2\alpha}n)\sigma(\frac{n}{\alpha}),   
\end{align}
\begin{align}  \label{eqn-royer_cooper_ye-10}
 W_{(1,10)}(n)  = & 
 \frac{1}{312}\sigma_{3}(n) 
  + \frac{1}{78}\sigma_{3}(\frac{n}{2})
  + \frac{25}{312}\sigma_{3}(\frac{n}{5})  
  + \frac{25}{78}\sigma_{3}(\frac{n}{10})  
  + (\frac{1}{24}-\frac{1}{40}n)\sigma(n)     \notag \\  & 
  + (\frac{1}{24}-\frac{1}{4}n)\sigma(\frac{n}{10})   
  - \frac{31}{1560}\EuFrak{b}_{40,1}(n)   
  - \frac{5}{52}\EuFrak{b}_{40,1}(\frac{n}{2})
  + \frac{1}{12}\EuFrak{b}_{40,3}(n), 
\end{align}
\begin{align}  \label{eqn-royer_cooper_ye-2-5}
 W_{(2,5)}(n)  = & 
 \frac{1}{312}\sigma_{3}(n) 
  + \frac{1}{78}\sigma_{3}(\frac{n}{2})
  + \frac{25}{312}\sigma_{3}(\frac{n}{5})  
  + \frac{25}{78}\sigma_{3}(\frac{n}{10})   
  + (\frac{1}{24}-\frac{1}{20}n)\sigma(\frac{n}{2})     \notag \\  & 
  + (\frac{1}{24}-\frac{1}{8}n)\sigma(\frac{n}{5})   
  - \frac{1}{312}\EuFrak{b}_{40,1}(n)           
  + \frac{1}{260}\EuFrak{b}_{40,1}(\frac{n}{2})
  - \frac{1}{12}\EuFrak{b}_{40,3}(n), 
\end{align}
\begin{align}
W_{(1,11)}(n)   = & 
 \frac{5}{1464}\,\sigma_{3}(n)
 + \frac{605}{1464}\,\sigma_{3}(\frac{n}{11}) 
 + (\frac{1}{24}-\frac{1}{44}n)\sigma(n)  \notag \\ & 
 + (\frac{1}{24}-\frac{1}{4}n)\sigma(\frac{n}{11})  
 - \frac{14615}{386496}\, \EuFrak{b}'_{33,1}(n)  
 - \frac{90493}{386496}\,\EuFrak{b}_{33,2}(n), 
\label{convolSum-theor-w_1_11}  
\end{align} 
\begin{align}  \label{eqn-aalaca_et_al-1_12}
 W_{(1,12)}(n)  = & 
   \frac{1}{480}\,\sigma_{3}(n) 
  + \frac{1}{160}\,\sigma_{3}(\frac{n}{2}) 
  + \frac{3}{160}\,\sigma_{3}(\frac{n}{3}) 
  + \frac{1}{30}\,\sigma_{3}(\frac{n}{4}) 
  + \frac{9}{160}\,\sigma_{3}(\frac{n}{6})     \notag \\  & 
  + \frac{3}{10}\,\sigma_{3}(\frac{n}{12}) 
  + (\frac{1}{24}-\frac{1}{48}n)\sigma(n) 
  + (\frac{1}{24}-\frac{1}{4}n)\sigma(\frac{n}{12})   
   - \frac{11}{480}\,\EuFrak{b}_{24,1}(n)      \notag \\  & 
   - \frac{17}{240}\,\EuFrak{b}_{24,2}(n)
   - \frac{1}{16}\,\EuFrak{b}_{24,3}(n),
\end{align}
\begin{align}  \label{eqn-aalaca_et_al-3_4}
 W_{(3,4)}(n)  = &  
   \frac{1}{480}\,\sigma_{3}(n) 
  + \frac{1}{160}\,\sigma_{3}(\frac{n}{2})
  + \frac{3}{160}\,\sigma_{3}(\frac{n}{3})
  + \frac{1}{30}\,\sigma_{3}(\frac{n}{4})
  + \frac{9}{160}\,\sigma_{3}(\frac{n}{6})      \notag \\  & 
  + \frac{3}{10}\,\sigma_{3}(\frac{n}{12})
  + (\frac{1}{24}-\frac{1}{16}n)\sigma(\frac{n}{3}) 
  + (\frac{1}{24}-\frac{1}{12}n)\sigma(\frac{n}{4})   
  - \frac{1}{480}\,\EuFrak{b}_{24,1}(n)      \notag \\  & 
  - \frac{7}{240}\,\EuFrak{b}_{24,2}(n)
  + \frac{1}{16}\,\EuFrak{b}_{24,3}(n),
\end{align}
\begin{align}  \label{eqn-royer_cooper_ye-1-15}
 W_{(1,15)}(n)  = & 
 \frac{1}{1560}\sigma_{3}(n)
 + \frac{1}{65}\sigma_{3}(\frac{n}{3})
 +  \frac{25}{39}\sigma_{3}(\frac{n}{5})
 - \frac{25}{104}\sigma_{3}(\frac{n}{15})  
  + (\frac{1}{24}-\frac{1}{60}n)\sigma(n)       \notag \\  & 
  + (\frac{1}{24}-\frac{1}{4}n)\sigma(\frac{n}{15})   
  - \frac{1}{39}\,\EuFrak{b}_{15,1}(n)           
  - \frac{2}{15}\,\EuFrak{b}_{15,2}(n)
 - \frac{14}{65}\,\EuFrak{b}_{15,3}(n)
 + \frac{1}{5}\,\EuFrak{b}_{15,4}(n),  
\end{align}
\begin{align}  \label{eqn-royer_cooper_ye-3-5}
 W_{(3,5)}(n)  = & 
  \frac{1}{390}\sigma_{3}(n)
 + \frac{7}{520}\sigma_{3}(\frac{n}{3})
 - \frac{175}{312}\sigma_{3}(\frac{n}{5})
 + \frac{25}{26}\sigma_{3}(\frac{n}{15})    
  + (\frac{1}{24}-\frac{1}{20}n)\sigma(\frac{n}{3})   \notag \\  & 
  + (\frac{1}{24}-\frac{1}{8}n)\sigma(\frac{n}{5})   
  - \frac{1}{390}\,\EuFrak{b}_{15,1}(n)           
  - \frac{1}{30}\,\EuFrak{b}_{15,2}(n)
  - \frac{1}{26}\,\EuFrak{b}_{15,3}(n)
 - \frac{1}{5}\,\EuFrak{b}_{15,4}(n), 
\end{align}
\begin{align}  \label{eqn-aalaca_et_al-24}
 W_{(1,24)}(n)  = & 
  - \frac{1}{64}\,\sigma_{3}(n) 
  - \frac{77}{320}\,\sigma_{3}(\frac{n}{2}) 
  + \frac{1}{48}\,\sigma_{3}(\frac{n}{3}) 
  + \frac{127}{480}\,\sigma_{3}(\frac{n}{4}) 
  + \frac{41}{160}\,\sigma_{3}(\frac{n}{6})    \notag \\  & 
  + \frac{1}{30}\,\sigma_{3}(\frac{n}{8}) 
  - \frac{97}{480}\,\sigma_{3}(\frac{n}{12}) 
  + \frac{3}{10}\,\sigma_{3}(\frac{n}{24})   
  + (\frac{1}{24}-\frac{1}{96}n)\sigma(n)       \notag \\  & 
  + (\frac{1}{24}-\frac{1}{4}n)\sigma(\frac{n}{24})   
  - \frac{1}{64}\,\EuFrak{b}_{24,1}(n)
  + \frac{51}{160}\,\EuFrak{b}_{24,2}(n)
  + \frac{5}{16}\,\EuFrak{b}_{24,3}(n)       \notag \\  & 
  + \frac{883}{240}\,\EuFrak{b}_{24,4}(n)   
  + 2\,\EuFrak{b}_{24,5}(n)
  + \frac{183}{16}\,\EuFrak{b}_{24,6}(n)   
  + 10\,\EuFrak{b}_{24,7}(n)      
  + 31\,\EuFrak{b}_{24,8}(n), 
\end{align}
\begin{align}  \label{eqn-aalaca_et_al-3_8}
 W_{(3,8)}(n)  = & 
  - \frac{1}{160}\,\sigma_{3}(\frac{n}{2}) 
  + \frac{1}{192}\,\sigma_{3}(\frac{n}{3}) 
  + \frac{7}{480}\,\sigma_{3}(\frac{n}{4})
  + \frac{7}{320}\,\sigma_{3}(\frac{n}{6})       \notag \\  & 
  + \frac{1}{30}\,\sigma_{3}(\frac{n}{8})
  + \frac{23}{480}\,\sigma_{3}(\frac{n}{12})
  + \frac{3}{10}\,\sigma_{3}(\frac{n}{24})      
  + (\frac{1}{24}-\frac{1}{32}n)\sigma(\frac{n}{3})      \notag \\  & 
  + (\frac{1}{24}-\frac{1}{12}n)\sigma(\frac{n}{8})   
  + \frac{1}{160}\,\EuFrak{b}_{24,2}(n)
  + \frac{3}{64}\,\EuFrak{b}_{24,3}(n)
  + \frac{13}{240}\,\EuFrak{b}_{24,4}(n)     \notag \\  & 
  + \frac{9}{16}\,\EuFrak{b}_{24,6}(n)
  + \EuFrak{b}_{24,8}(n).  
\end{align}
\end{corollary}
For example \autoref{besge-glaisher-ramanujan-1} is easily proved as follows. 
Due to \autoref{evalConvolClass-lema-1-eqn-0} and applying 
\autoref{evalConvolClass-eqn-11} we have 
 \begin{equation*} 
 0 = 
 - 1152\,\alpha^{2}\,W_{(\alpha,\alpha)}(n) 
 + 480\,\alpha^{2}\,\sigma_{3}(\frac{n}{\alpha}) 
    + 96\,\alpha\,(\alpha-6\,n)\,\sigma(\frac{n}{\alpha}).   
\end{equation*}
Therefore, we obtain \autoref{convolutionSum-a-a-eqn}. By setting $\alpha=1$, one 
gets the result obtained by \mB\ \cite{besge}, \jwlG\ \cite{glaisher} and \sR\ 
\cite{ramanujan}. 


\section{Formulae for the Number of Representations of a positive Integer}  
\label{representations}

We make use of the convolution sums evaluated in 
\hyperref[convolution_33_40_56]{Section \ref*{convolution_33_40_56}}
among others to determine 
explicit formulae for the number of representations of a positive
integer $n$  by the octonary quadratic forms \autoref{introduction-eq-1} 
and \autoref{introduction-eq-2}, respectively.

\subsection{Number of Representations of a positive Integer applying illustrated 
Convolution Sums} 
\label{representations_33_40_56}

\subsubsection{Representations by the Octonary Quadratic Forms \autoref{introduction-eq-1}}

We determine formulae for the number of representations of a positive integer $n$ 
by the Octonary Quadratic Form \autoref{introduction-eq-2} when we mainly apply the 
evaluation of the convolution sums $W_{(1,33)}(n)$ and $W_{(3,11)}(n)$. In order to  
do that, we recall that $33=3\cdot 11$ is of the restricted form in 
\hyperref[representations_c_d]{Section \ref*{representations_c_d}}. Hence, 
from \hyperref[representation-prop-2]{Proposition \ref*{representation-prop-2}} we 
derive that $\Omega_{3}=\{(1,11)\}$. 
We then deduce the following result:
\begin{corollary} \label{representations-coro_33}
Let $n\in\mathbb{N}$. Then  
\begin{align*}
R_{(1,11)}(n)  = & 
12\sigma(n) - 36\sigma(\frac{n}{3}) + 12\sigma(\frac{n}{11}) -
36\sigma(\frac{n}{33}) + 144\, W_{(1,11)}(n) \\ &
+ 1296\, W_{(1,11)}(\frac{n}{3}) 
 - 432\, \biggl( W_{(3,11)}(n) + W_{(1,33)}(n) \biggr).
\end{align*}
\end{corollary}
\begin{proof} 
It follows immediately from \autoref{representations-theor_a_b} 
with $(c,d)=(1,11)$.  One can then make 
use of \autoref{convolSum-theor-w_1_11}, 
\autoref{convolSum-theor-w_1_33} and \autoref{convolSum-theor-w_3_11}
 to simplify this formula. 
\end{proof}

\subsubsection{Representations by Octonary Quadratic Forms \autoref{introduction-eq-2}}

We give formulae for the number of representations of a positive integer $n$ 
by the Octonary Quadratic Form \autoref{introduction-eq-1} by mainly applying the 
evaluation of the convolution sums $W_{(1,40)}(n)$, $W_{(1,56)}(n)$, $W_{(5,8)}(n)$ 
and $W_{(7,8)}(n)$. To achieve  
that, we recall that $40=2^{3}\cdot 5$ and $56=2^{3}\cdot 7$ are of the restricted 
form in 
\hyperref[representations_a_b]{Section \ref*{representations_a_b}}. Therefore, 
we apply \hyperref[representation-prop-1]{Proposition \ref*{representation-prop-1}} 
to conclude that $\Omega_{4}=\{(1,10), (2,5)\}$ in case $\alpha\beta=40$ and 
$\Omega_{4}=\{(1,14), (2,7)\}$ in case $\alpha\beta=56$. 

\begin{corollary} \label{representations-coro-40_56}
Let $n\in\mathbb{N}$. Then  
\begin{align*}
N_{(1,10)}(n)  &  = 
8\sigma(n) - 32\sigma(\frac{n}{4}) + 8\sigma(\frac{n}{10}) -
32\sigma(\frac{n}{40}) + 64\, W_{(1,10)}(n) 
   + 1024\, W_{(1,10)}(\frac{n}{4}) \\ &
   - 256\, \biggl( W_{(2,5)}(\frac{n}{2}) +  W_{(1,40)}(n) \biggr), 
\end{align*}
\begin{align*}
N_{(2,5)}(n) & = ~ 
8\sigma(\frac{n}{2}) - 32\sigma(\frac{n}{8}) + 8\sigma(\frac{n}{5}) -
32\sigma(\frac{n}{20}) + 64\, W_{(2,5)}(n) 
   + 1024\, W_{(2,5)}(\frac{n}{4}) \\ & 
   - 256\, \biggl( W_{(5,8)}(n) +  W_{(1,10)}(\frac{n}{2}) \biggr), 
\end{align*}
\begin{align*}
N_{(1,14)}(n)  & = ~ 
8\sigma(n) - 32\sigma(\frac{n}{4}) + 8\sigma(\frac{n}{14}) -
32\sigma(\frac{n}{56}) 
 + 64\, W_{(1,14)}(n) + 1024\, W_{(1,14)}(\frac{n}{4}) \\ &
 - 256\, \biggl( W_{(2,7)}(\frac{n}{2}) + W_{(1,56)}(n) \biggr),  
\end{align*}
\begin{align*}
N_{(2,7)}(n)  & = ~  
8\sigma(\frac{n}{2}) - 32\sigma(\frac{n}{8}) 
+ 8\sigma(\frac{n}{7}) - 32\sigma(\frac{n}{28}) 
 + 64\, W_{(2,7)}(n) + 1024\, W_{(2,7)}(\frac{n}{4})   \\ & 
 - 256\, \biggl( W_{(7,8)}(n) + W_{(1,14)}(\frac{n}{2}) \biggr).
\end{align*}
\end{corollary}
\begin{proof} 
These formulae follow immediately from \autoref{representations-theor_a_b} 
when we set $(a,b)=(1,10)$, $(2,5)$, $(1,14)$, $(2,7)$, respectively. 
One can then use the result of 
\begin{itemize}
\item  \sC\ and \dY\ \cite[Thrm 2.1]{cooper_ye2014},   
\autoref{eqn-royer_cooper_ye-10}, 
\autoref{convolSum-theor-w_1_40} and \autoref{convolSum-theor-w_5_8}  
for the sake of simplification in case of $N_{10}$ and $N_{(2,5)}$.  
\item \eR\ \cite[Thrms 1.7]{royer}, \eN\ \cite[Thrm 3.2.1]{ntienjem2015},  
\autoref{convolSum-theor-w_1_56} and \autoref{convolSum-theor-w_7_8}
 to simplify the formulae in case of $N_{(1,14)}$ and $N_{(2,7)}$.
\end{itemize}
\end{proof}


\subsection{Revisited Formulae for the Number of representations of a positive integer}
\label{residue-representation}

In the following subsection, formulae for the number of representations of a positive 
integer $n$, $N_{(a,b)}(n)$, for $(a,b)=(1,1),(1,3),(2,3),(1,9)$, are determined as  
applications of the evaluation of the convolution sums 
$W_{(1,4)}(n)$ by \jgH\ et al.\ \cite{huardetal}, 
$W_{(1,12)}(n)$, $W_{(3,4)}(n)$, $W_{(1,24)}(n)$ and $W_{(3,8)}(n)$ by 
\aA\ et al.\ \cite{alaca_alaca_williams2006,alaca_alaca_williams2007a}, and
$W_{(1,36)}(n)$ and $W_{(4,9)}(n)$ by \dY\ \cite{ye2015}. These numbers of 
representations of a positive integer $n$ are discovered due to 
\hyperref[representation-prop-1]{Proposition \ref*{representation-prop-1}}. We 
rather consider \autoref{eqn-aalaca_et_al-1_12}, \autoref{eqn-aalaca_et_al-3_4}, 
\autoref{eqn-aalaca_et_al-24} and \autoref{eqn-aalaca_et_al-3_8} in the following 
result. 
\begin{corollary} \label{representations-coro-residue}
Let $n\in\mathbb{N}$. Then  
\begin{align*}
N_{(1,1)}(n) & = 16\,\sigma(n) - 64\,\sigma(\frac{n}{4})  
 + 64\, W_{(1,1)}(n) + 1024\, W_{(1,1)}(\frac{n}{4}) - 512\,W_{(1,4)}(n) \\ & 
   = 16\,\sigma_{3}(n) - 32\,\sigma_{3}(\frac{n}{2}) + 256\,\sigma_{3}(\frac{n}{4}),
\end{align*}
\begin{align*}
N_{(1,3)}(n) & = 8\sigma(n) - 32\sigma(\frac{n}{4}) + 8\sigma(\frac{n}{3}) -
32\sigma(\frac{n}{12}) 
 + 64\, W_{(1,3)}(n) + 1024\, W_{(1,3)}(\frac{n}{4}) \\ &
 - 256\, \biggl( W_{(3,4)}(n) + W_{(1,12)}(n) \biggr), 
\end{align*}
\begin{align*}
N_{(2,3)}(n)  & = 8\,\sigma(\frac{n}{2}) - 32\,\sigma(\frac{n}{8}) 
+ 8\,\sigma(\frac{n}{3}) - 32\,\sigma(\frac{n}{12}) 
 + 64\, W_{(1,3)}(n) + 1024\, W_{(1,3)}(\frac{n}{4}) \\ &
 - 256\, \biggl( W_{(3,8)}(n) + W_{(1,12)}(n) \biggr), 
\end{align*}
\begin{align*}
N_{(1,9)}(n)  & = 8\,\sigma(n) - 32\,\sigma(\frac{n}{4}) 
+ 8\,\sigma(\frac{n}{9}) - 32\,\sigma(\frac{n}{36}) 
 + 64\,W_{(1,9)}(n) + 1024\,W_{(1,9)}(\frac{n}{4}) \\ & 
 - 256\,\biggl( W_{(4,9)}(n) + W_{(1,36)}(n) \biggr), 
\end{align*}
\end{corollary}
\begin{proof} 
When we set $(a,b)=(1,1)$, $(1,3)$, $(2,3)$, $(1,9)$, these formulae follow 
immediately from \autoref{representations-theor_a_b}. 
\end{proof}


\section{Concluding Remark} \label{conclusion}
To evaluate the convolution sum for $\alpha\beta$ 
that belongs to this class of positive integers, it now suffices to 
determine a basis of the space of cusp forms of weight $4$ for
$\Gamma_{0}(\alpha\beta)$. It is straightforward from \autoref{basisCusp_a_b} (a), 
that a basis of the space of Eisenstein forms of weight $4$ for 
$\Gamma_{0}(\alpha\beta)$ is already given. 

The determination of a basis of the space of cusp forms when 
$\alpha\beta$ is large is tedious. A future work is to carry out an effective 
and efficient method to build a basis of a space of cusp forms of weight 
$4$ for $\Gamma_{0}(\alpha\beta)$ when $\alpha\beta$ is large.  

A natural number can be expressed as a finite product of distinct primes to 
the power of some positive integers. The 
form for $\alpha\beta$ that we have considered falls under such an 
expression; however that form does not cover all natural numbers.  
The consideration of the class of natural numbers which is not discussed in 
this paper is a work in progress.




\section*{Competing interests}
  The authors declare that they have no competing interests.

\section*{Acknowledgments}  
I am indebtedly thankful to 
Prof.\ Emeritus Kenneth S. Williams 
for fruitful comments and suggestions on a draft of this paper.

\section*{Tables}

\begin{longtable}{|c|cccc|} \hline
   & \textbf{1}  &  \textbf{3}  & \textbf{11} & \textbf{33}  \\ \hline
\textbf{1}  & 0 & 8 & 0 & 0   \\ \hline
\textbf{2}  & 4 & 0 & 4 & 0  \\ \hline
\textbf{3}  & 3 & 1 & 3 & 1  \\ \hline
\textbf{4}  & 2 & 2 & 2 & 2   \\ \hline
\textbf{5}  & 1 & 3 & 1 & 3   \\ \hline
\textbf{6}  & 0 & 4 & 0 & 4   \\ \hline
\textbf{7}  & -1 & 5 & -1 & 5 \\ \hline
\textbf{8}  & -2 & 6 & -2 & 6   \\ \hline
\textbf{9}  & 6 & 0 & 0 & 2  \\ \hline
\textbf{10} & 4 & -2 & -2 & 8  \\ \hline
\caption{Power of $\eta$-quotients being basis elements of $\S_{4}(\Gamma_{0}(33))$}
\label{convolutionSums-3_11-table}
\end{longtable}

{ 
\begin{longtable}{|c|cccccccc|} \hline
   & \textbf{1}  &  \textbf{2}  & \textbf{4} & \textbf{5}  &
   \textbf{8} & \textbf{10} & \textbf{20}  & \textbf{40}  \\ \hline
 \textbf{1}  & 4 & 0 & 0 & 4 & 0 & 0 & 0 & 0  \\ \hline
 \textbf{2}  & 0 & 4 & 0 & 0 & 0 & 4 & 0 & 0   \\ \hline
 \textbf{3}  & 2 & 0 & 0 & -2 & 0 & 8 & 0 & 0    \\ \hline
 \textbf{4}  & 0 & 0 & 4 & 0 & 0 & 0 & 4 & 0   \\ \hline
 \textbf{5}  & 0 & 0 & 0 & 0 & 0 & 4 & 4 & 0   \\ \hline
 \textbf{6} & 0 & 2 & 0 & 0 & 0 & -2 & 8 & 0    \\ \hline
 \textbf{7} & 2 & -2 & 0 & -2 & 0 & 2 & 8 & 0   \\ \hline
 \textbf{8}  & 0 & 0 & 0 & 0 & 4 & 0 & 0 & 4   \\ \hline
 \textbf{9}  & 0 & 0 & 0 & 0 & 2 & 4 & -4 & 6    \\ \hline
 \textbf{10}  & 2 & -2 & 2 & 2 & -2 & 0 & 0 & 6   \\ \hline
 \textbf{11}  & 1 & 0 & 0 & -1 & 1 & 2 & -2 & 7  \\ \hline
 \textbf{12} & 0 & 0 & 2 & 0 & 0 & 0 & -2 & 8    \\ \hline
 \textbf{13} & 0 & 4 & 0 & 0 & -2 & 0 & -4 & 10    \\ \hline
 \textbf{14} & 0 & 2 & -2 & 0 & 0 & -2 & 2 & 8   \\ \hline
\caption{Power of $\eta$-quotients being basis elements of $\S_{4}(\Gamma_{0}(40))$}
\label{convolutionSums-5_8-table}
\end{longtable}
}

\begin{longtable}{|c|cccccccc|} \hline
   & \textbf{1}  &  \textbf{2}  & \textbf{4} & \textbf{7} & \textbf{8}  &
   \textbf{14}  & \textbf{28} & \textbf{56}  \\ \hline
\textbf{1}  & 5 & -1 & 0 & 5 & 0 & -1 & 0 & 0  \\ \hline
\textbf{2}  & 2 & 2 & 0 & 2 & 0 & 2 & 0 & 0 \\ \hline
\textbf{3}  & 6 & -2 & 0 & -2 & 0 & 6 & 0 & 0  \\ \hline
\textbf{4}  & 0 & 2 & 2 & 0 & 0 & 2 & 2 & 0  \\ \hline
\textbf{5}  & 0 & 0 & 2 & 0 & 0 & 4 & 2 & 0  \\ \hline
\textbf{6}  & 0 & 6 & -2 & 0 & 0 & -2 & 6 & 0  \\ \hline
\textbf{7}  & 0 & 4 & -2 & 0 & 0 & 0 & 6 & 0  \\ \hline
\textbf{8}  & 1 & 1 & 0 & 1 & 0 & -3 & 8 & 0  \\ \hline
\textbf{9}  & 0 & 1 & 1 & 0 & 0 & -3 & 9 & 0  \\ \hline
\textbf{10} & 0 & 0 & 0 & 0 & 2 & 0 & 4 & 2  \\ \hline
\textbf{11} & 0 & -2 & 8 & 0 & -2 & 2 & -4 & 6  \\ \hline
\textbf{12} & 0 & 0 & 6 & 0 & -2 & 0 & -2 & 6  \\ \hline
\textbf{13} & 0 & 0 & 3 & 0 & -1 & 4 & -5 & 7  \\ \hline
\textbf{14} & 0 & 0 & 4 & 0 & -2 & 0 & 0 & 6  \\ \hline
\textbf{15} & 0 & 2 & 2 & 0 & -2 & -2 & 2 & 6  \\ \hline
\textbf{16} & 0 & 1 & 1 & 0 & 0 & 1 & -3 & 8  \\ \hline
\textbf{17} & 0 & 3 & -1 & 0 & 0 & -1 & -1 & 8  \\ \hline
\textbf{18} & 0 & 0 & 1 & 0 & 1 & 0 & -3 & 9  \\ \hline
\textbf{19} & 0 & 1 & 0 & 0 & -1 & -3 & 4 & 7  \\ \hline
\textbf{20} & -2 & 5 & -3 & 2 & 0 & -5 & 7 & 4  \\ \hline
\caption{Power of $\eta$-quotients being basis elements of $\S_{4}(\Gamma_{0}(56))$}
\label{convolutionSums-7_8-table}
\end{longtable}





\begin{thebibliography}{10}
\bibitem{glaisher}
J.~W.~L. Glaisher.
\newblock On the square of the series in which the coefficients are the sums of
  the divisors of the exponents.
\newblock {\em Messenger Math}, 14:156--163, 1862.
\end{thebibliography}
\end{document}